\documentclass[a4paper,reqno,10pt]{amsart}

\usepackage{epsf,graphicx,epsfig}
\usepackage{amscd}
\usepackage{amsmath,latexsym,amssymb,amsthm}
\usepackage[nospace,noadjust]{cite}
\usepackage{textcomp}
\usepackage{dsfont}
\usepackage{fdsymbol}
\usepackage{setspace,cite}
\usepackage{lscape,fancyhdr,fancybox}
\usepackage{stmaryrd}
\usepackage[all,cmtip]{xy}
\usepackage{tikz}
\usepackage{cancel}
\usetikzlibrary{shapes,arrows,decorations.markings}
\usepackage{graphicx}

\usepackage{color}
\usepackage{url}
\usepackage{enumerate}
\usepackage[mathscr]{euscript}

  \theoremstyle{plain}

\swapnumbers
    \newtheorem{thm}{Theorem}[section]
    \newtheorem{prop}[thm]{Proposition}
     
   \newtheorem{lemma}[thm]{Lemma}

    \newtheorem{subsec}[thm]{}
\theoremstyle{definition}
    \newtheorem{defn}[thm]{Definition}
        \newtheorem{remark}[thm]{Remark}
    \newtheorem{exam}[thm]{Example}

\theoremstyle{remark}

\title{}
\author{}
\date{}
\usepackage{amssymb}

\usepackage{hyperref}
\hypersetup{
	colorlinks,
	citecolor=blue,
	filecolor=black,
	linkcolor=blue,
	urlcolor=black
}

\begin{document}

\title[Representations and cohomology of a matched pair of Lie algebras]{Representations and cohomology of a matched pair of Lie algebras, and $L_\infty$-algebras}

\author{Anusuiya Baishya}
\address{Department of Mathematics,
Indian Institute of Technology, Kharagpur 721302, West Bengal, India.}
\email{anusuiyabaishya530@gmail.com}

\author{Apurba Das}
\address{Department of Mathematics,
Indian Institute of Technology, Kharagpur 721302, West Bengal, India.}
\email{apurbadas348@gmail.com, apurbadas348@maths.iitkgp.ac.in}


\begin{abstract}
The notion of a matched pair of Lie algebras was introduced in the study of Lie bialgebras and Poisson-Lie groups. In this paper, we introduce representations and cohomology of a matched pair of Lie algebras. We show that there is a morphism from the cohomology of a Lie bialgebra to the cohomology of the corresponding matched pair of Lie algebras. Our cohomology is also useful to study infinitesimal deformations and abelian extensions. In the last part, we define the notion of a matched pair of $L_\infty$-algebras and construct the corresponding bicrossed product. Finally, we show that a matched pair of skeletal $L_\infty$-algebras is closely related to the cohomology introduced in the paper.
\end{abstract}

\maketitle



\medskip

\medskip

\medskip






\thispagestyle{empty}

\tableofcontents


\medskip

\section{Introduction}\label{sec1}
The notion of a Lie bialgebra was introduced by Drinfel'd \cite{drinfeld} in the study of the algebraic realization of a Poisson-Lie group. A Lie bialgebra is a pair $(\mathfrak{g}, \mathfrak{g}^*)$ consisting of Lie algebras $\mathfrak{g}$ and $\mathfrak{g}^*$ such that the map $\delta: \mathfrak{g} \rightarrow \wedge^2 \mathfrak{g}$ (dualizing the Lie bracket of $\mathfrak{g}^*$) is a $1$-cocycle of the Lie algebra $\mathfrak{g}$ with coefficients in the adjoint representation in $\wedge^2 \mathfrak{g}$. Then it turns out that the direct sum $\mathfrak{g} \oplus \mathfrak{g}^*$ inherits a new Lie algebra structure (called the double of the Lie bialgebra) for which $\mathfrak{g}$ and $\mathfrak{g}^*$ are both Lie subalgebras. The abstraction of the structures of $\mathfrak{g} \oplus \mathfrak{g}^*$ leads to the concept of a matched pair of Lie algebras (also called double Lie algebras \cite{lu-weins}, or twilled extensions of Lie algebras \cite{koss}). Precisely, a matched pair of Lie algebras is a quadruple $(\mathfrak{g}, \mathfrak{h}, \rho, \psi)$ in which $\mathfrak{g}$, $\mathfrak{h}$ are both Lie algebras, $\rho: \mathfrak{g} \times \mathfrak{h} \rightarrow \mathfrak{h}$ and $\psi:\mathfrak{h} \times \mathfrak{g} \rightarrow \mathfrak{g}$ are actions of Lie algebras satisfying some compatibility conditions (see Definition \ref{defn-mpl} for details). If $(\mathfrak{g}, \mathfrak{g}^*)$ is a Lie bialgebra then the quadruple $(\mathfrak{g}, \mathfrak{g}^*, \mathrm{ad}^*_\mathfrak{g}, \mathrm{ad}^*_{\mathfrak{g}^*})$ forms a matched pair of Lie algebras, where $\mathrm{ad}^*_\mathfrak{g} : \mathfrak{g} \times \mathfrak{g}^* \rightarrow \mathfrak{g}^*$ and $\mathrm{ad}^*_{\mathfrak{g}^*} : \mathfrak{g}^* \times \mathfrak{g} \rightarrow \mathfrak{g}$ are the coadjoint actions. Given a matched pair of Lie algebras $(\mathfrak{g}, \mathfrak{h}, \rho, \psi)$, one can define the bicrossed product Lie algebra $\mathfrak{g} \Join \mathfrak{h}$ for which $\mathfrak{g}$ and $\mathfrak{h}$ are both Lie subalgebras. A matched pair of Lie algebras can also be characterized by a triple $(\mathfrak{g} \oplus \mathfrak{h}, \mathfrak{g}, \mathfrak{h})$, where $\mathfrak{g} \oplus \mathfrak{h}$ carries a Lie algebra structure for which $\mathfrak{g}$ and $\mathfrak{h}$ are both Lie subalgebras.

\medskip

It is important to remark that the notion of a matched pair of Lie algebras has been generalized to various contexts. In \cite{etingof} the authors have considered the notion of a matched pair of groups that are closely related to braided groups and the set-theoretical solutions of the Yang-Baxter equation. The concept of a matched pair of Leibniz algebras was defined in \cite{agore} to study the problem of extending structures. In the geometric context, matched pairs of Lie groupoids and Lie algebroids are extensively studied in \cite{mack,mokri}.

\medskip

In the first part of this paper, we take the initiative to study representations and cohomology of a matched pair of Lie algebras. A representation of a matched pair of Lie algebras $(\mathfrak{g}, \mathfrak{h}, \rho, \psi)$ is given by a quadruple $(V, W, \alpha, \beta)$ in which $V$ and $W$ are both representations of the Lie algebras $\mathfrak{g}$ and $\mathfrak{h}$, and $\alpha: V \times \mathfrak{h} \rightarrow W$ and $\beta: W \times \mathfrak{g} \rightarrow V$ are bilinear maps (called the pairing maps) satisfying a set of identities (see Definition \ref{defn-mpl-rep}). It turns out that a matched pair of Lie algebras $(\mathfrak{g}, \mathfrak{h}, \rho, \psi)$ can be realized as a representation of itself, called the adjoint representation. We also give the semidirect product construction in the context of matched pairs of Lie algebras. Next, given two vector spaces $\mathfrak{g}$ and $\mathfrak{h}$ (not necessarily equipped with any additional structures), we construct a graded Lie algebra whose Maurer-Cartan elements correspond to matched pairs of Lie algebra structures on the pair of vector spaces $(\mathfrak{g}, \mathfrak{h})$. This characterization is useful to define the cohomology of a matched pair of Lie algebras with coefficients in the adjoint representation. We show that there is a morphism from the cohomology of a matched pair of Lie algebras $(\mathfrak{g}, \mathfrak{h}, \rho, \psi)$ to the Chevalley-Eilenberg cohomology of the bicrossed product Lie algebra $\mathfrak{g} \Join \mathfrak{h}$. We also find a morphism from the cohomology of a Lie bialgebra $(\mathfrak{g}, \mathfrak{g}^*)$ to the cohomology of the corresponding matched pair of Lie algebras $(\mathfrak{g}, \mathfrak{g}^*, \mathrm{ad}^*_\mathfrak{g}, \mathrm{ad}^*_{\mathfrak{g}^*})$. Finally, we also define the cohomology with coefficients in arbitrary representation. As applications of our cohomology theories, we study infinitesimal deformations and abelian extensions of a matched pair of Lie algebras.

\medskip

On the other hand, an $L_\infty$-algebra (also called a strongly homotopy Lie algebra) is a generalization of Lie algebra where the Jacobi identity holds up to coherent homotopy. They were first introduced by Lada and Stasheff \cite{lada-stasheff}, and further studied by Lada and Markl \cite{lada-markl}. $L_\infty$-algebras has many applications in mathematics and physics including deformation theory \cite{doubek} (specifically, deformation quantization of Poisson manifolds \cite{kont}), classical field theory \cite{rogers} and higher differential geometry \cite{rogers,roy}. Any Lie algebra, graded Lie algebra and differential graded Lie algebra can be seen as particular examples of $L_\infty$-algebras. In \cite{baez-crans} Baez and Crans considered those $L_\infty$-algebras whose underlying graded vector spaces are concentrated in two degrees. They call them $2$-term $L_\infty$-algebras. Among others, they considered skeletal $L_\infty$-algebras (which are $2$-term $L_\infty$-algebras with zero differential) and showed that they are closely related to Chevalley-Eilenberg $3$-cocycles of Lie algebras.

\medskip

In another part of this paper, we introduce the notion of a matched pair of $L_\infty$-algebras. Any matched pair of Lie algebras is an example of a matched pair of $L_\infty$-algebras. As a particular case of a matched pair of $L_\infty$-algebras, we obtain a matched pair of graded Lie algebras. Given a Lie algebra $\mathfrak{g}$, we show that the Nijenhuis-Richardson graded Lie algebra acts on the cup-product Lie algebra by derivations. Hence we obtain an example of a matched pair of graded Lie algebras. 
Next, we also find a Maurer-Cartan characterization of a matched pair of $L_\infty$-algebras. We also show that a matched pair of $L_\infty$-algebras give rise to a bicrossed product $L_\infty$-algebra. This generalizes the well-known bicrossed product in the homotopy context. In \cite{zhu} the authors have shown that a Rota-Baxter operator of weight $1$ on a Lie algebra $\mathfrak{g}$ gives rise to a matched pair of Lie algebras. We generalize their result to those $L_\infty$-algebras $(\mathcal{G}, \{ \mu_k \}_{k \geq 1})$ in which $\mu_{k \geq 4} = 0$. Finally, we put our attention to matched pairs of skeletal $L_\infty$-algebras. We show that a matched pair of skeletal $L_\infty$-algebras is closely related to certain $3$-cocyles of matched pairs of Lie algebras.

\medskip

The paper is organized as follows. In Section \ref{sec2}, we recall the Chevalley-Eilenberg cohomology of a Lie algebra and revise the notion of a matched pair of Lie algebras. In Sections \ref{sec3} and \ref{sec4}, we respectively study representations and cohomology of a matched pair of Lie algebras. In particular, we construct the graded Lie algebra whose Maurer-Cartan elements correspond to matched pairs of Lie algebras on a given pair of vector spaces. As applications of our cohomology, in Section \ref{sec5}, we study infinitesimal deformations and abelian extensions of a matched pair of Lie algebras. In Section \ref{sec6}, we introduce the notion of a matched pair of $L_\infty$-algebras and define the bicrossed product. Finally, in Section \ref{sec7}, we consider matched pairs of skeletal $L_\infty$-algebras and characterize them in terms of cohomology.

\medskip

All vector spaces, graded vector spaces, (multi)linear maps, tensor products and wedge products are over a field ${\bf k}$ of characteristic zero.

\section{Lie algebras and matched pair of Lie algebras}\label{sec2}
In this section, we recall some necessary background on Lie algebras and matched pairs of Lie algebras. Among others, we describe the Chevalley-Eilenberg cohomology of a Lie algebra and the bicrossed product of a matched pair of Lie algebras.

Let $(\mathfrak{g}, [~, ~]_\mathfrak{g})$ be a Lie algebra. That is, $\mathfrak{g}$ is a vector space equipped with a skew-symmetric bilinear operation (called the bracket) $[~,~]_\mathfrak{g}: \mathfrak{g} \times \mathfrak{g} \rightarrow \mathfrak{g}$ satisfying the Jacobi identity:
\begin{align*}
    [[x, y]_\mathfrak{g} , z]_\mathfrak{g} +  [[ y, z]_\mathfrak{g} , x]_\mathfrak{g} +  [[z, x]_\mathfrak{g} , y]_\mathfrak{g} = 0, \text{ for all } x, y, z \in \mathfrak{g}.
\end{align*}
The Lie algebra  $(\mathfrak{g}, [~, ~]_\mathfrak{g})$ is often denoted simply by $\mathfrak{g}$ when the bracket is clear from the context. A {\em representation} of the Lie algebra  $(\mathfrak{g}, [~, ~]_\mathfrak{g})$ is a vector space $V$ together with a bilinear map (called the action map) $\rho: \mathfrak{g} \times V \rightarrow V$, $(x, v) \mapsto \rho_x v$ that satisfies
\begin{align*}
    \rho_{[x, y]_\mathfrak{g}} v = \rho_x \rho_y v - \rho_y \rho_x v, \text{ for } x, y \in \mathfrak{g}, v \in V.
\end{align*}
A representation as above may be denoted by $(V, \rho)$ or simply by $V$ when the action is understood.
Note that the Lie algebra $\mathfrak{g}$ can be realized as a representation of itself with the action map $\rho: \mathfrak{g} \times \mathfrak{g} \rightarrow \mathfrak{g}$ being the Lie bracket $[~,~]_\mathfrak{g}$. This is called the {\em adjoint representation}.

Let $(\mathfrak{g}, [~, ~]_\mathfrak{g})$ be a Lie algebra and $V$ be a representation of it with the action map $\rho : \mathfrak{g} \times V \rightarrow V$. Then the Chevalley-Eilenberg cochain complex of the Lie algebra $\mathfrak{g}$ with coefficients in the representation $V$ is given by $\{ C^\bullet (\mathfrak{g}; V), \delta_\mathrm{CE} \}$, where 
\begin{align*}
C^{n \geq 0} (\mathfrak{g}; V) := \mathrm{Hom} (\wedge^n \mathfrak{g}, V) = \{ f : \underset{n \text{ copies}}{\mathfrak{g} \times \cdots \times \mathfrak{g}} \rightarrow V |~ f \text{ is skew-symmetric and multilinear}\}.
\end{align*}
The coboundary map $\delta_\mathrm{CE} : C^n (\mathfrak{g}; V) \rightarrow C^{n+1} (\mathfrak{g} ; V)$ is given by
\begin{align*}
    ( \delta_\mathrm{CE} (f)) (x_1, \ldots, x_{n+1}) =~& \sum_{i=1}^{n+1} (-1)^{i+1} \rho_{x_i} f (x_1, \ldots, \widehat{x_i}, \ldots, x_{n+1}) \\
   ~& + \sum_{1 \leq i < j \leq n+1} (-1)^{i+j} f ([x_i, x_j]_\mathfrak{g}, x_1, \ldots, \widehat{x_i}, \ldots, \widehat{x_j}, \ldots, x_{n+1}),
\end{align*}
for $f \in C^n (\mathfrak{g} ; V)$ and $x_1, \ldots, x_{n+1} \in \mathfrak{g}$. The corresponding cohomology groups are called the Chevalley-Eilenberg cohomology of the Lie algebra $\mathfrak{g}$ with coefficients in the representation $V$. The cohomology groups are usually denoted by $H^\bullet_\mathrm{CE} (\mathfrak{g}; V)$.

Let $\mathfrak{g}$ be a vector space (not necessarily equipped with any additional structure). For each $m, n \geq 0$, the {\em Nijenhuis-Richardson bracket} 
\begin{align*}
    [~,~]_\mathrm{NR} : \mathrm{Hom} (\wedge^{m+1} \mathfrak{g} , \mathfrak{g}) \times \mathrm{Hom} (\wedge^{n+1} \mathfrak{g} , \mathfrak{g}) \rightarrow \mathrm{Hom} (\wedge^{m+n+1} \mathfrak{g} , \mathfrak{g})
\end{align*}
is given by
\begin{align}\label{nr-br}
    [f, g]_\mathrm{NR} = i_f g - (-1)^{mn} i_g f, \text{ where }
\end{align}
\begin{align*}
    (i_f g) (x_1, \ldots, x_{m+n+1} ) := \sum_{\sigma \in \mathrm{Sh} (m+1, n)} (-1)^\sigma g \big( f( x_{\sigma (1)}, \ldots,  x_{\sigma (m+1)}  ), x_{\sigma (m+2)}, \ldots, x_{\sigma (m+n+1)}   \big),
\end{align*}
for $f \in  \mathrm{Hom} (\wedge^{m+1} \mathfrak{g} , \mathfrak{g})$, $g \in  \mathrm{Hom} (\wedge^{n+1} \mathfrak{g} , \mathfrak{g})$ and $x_1, \ldots, x_{m+n+1} \in \mathfrak{g}$. The importance of the Nijenhuis-Richardson bracket is given by the following result  \cite{nij-ric}.

\begin{thm}
    (Nijenhuis-Richardson) Let $\mathfrak{g}$ be a vector space. 

    (i) Then the Nijenhuis-Richardson bracket $[~,~]_\mathrm{NR}$ makes the graded space $\oplus_{n \geq 0} \mathrm{Hom} (\wedge^{n+1} \mathfrak{g}, \mathfrak{g})$ of all skew-symmetric multilinear maps on $\mathfrak{g}$ into a graded Lie algebra. 

    (ii) A skew-symmetric bilinear map/bracket $\mu = [~,~]_\mathfrak{g} : \mathfrak{g} \times \mathfrak{g} \rightarrow \mathfrak{g}$ defines a Lie algebra structure on the vector space $\mathfrak{g}$ if and only if $\mu \in \mathrm{Hom}(\wedge^2 \mathfrak{g}, \mathfrak{g})$ is a Maurer-Cartan element of the graded Lie algebra $( \oplus_{n \geq 0} \mathrm{Hom} (\wedge^{n+1} \mathfrak{g}, \mathfrak{g}) , [~,~]_\mathrm{NR})$. 
\end{thm}

\begin{remark}
    Let $(\mathfrak{g}, [~,~]_\mathfrak{g})$ be a Lie algebra. Then the Chevalley-Eilenberg coboundary map $\delta_\mathrm{CE} : C^n (\mathfrak{g}; \mathfrak{g}) \rightarrow C^{n+1} (\mathfrak{g}; \mathfrak{g})$ of the Lie algebra $\mathfrak{g}$ with coefficients in the adjoint representation is simply given by
    $\delta_\mathrm{CE} (f) = - [\mu, f]_\mathrm{NR}$, for $f \in C^n (\mathfrak{g}; \mathfrak{g}) = \mathrm{Hom} (\wedge^n \mathfrak{g}, \mathfrak{g})$. Here $\mu \in \mathrm{Hom}(\wedge^2 \mathfrak{g}, \mathfrak{g})$ represents the Lie bracket of $\mathfrak{g}$.
\end{remark}

\medskip


In the following, we recall matched pairs of Lie algebras and the corresponding bicrossed product.

\begin{defn}\label{defn-mpl}
    A {\em matched pair of Lie algebras} is a quadruple $(\mathfrak{g}, \mathfrak{h}, \rho, \psi)$ in which $\mathfrak{g}, \mathfrak{h}$ are both Lie algebras together with bilinear maps $\rho : \mathfrak{g} \times \mathfrak{h} \rightarrow \mathfrak{h}$ and $\psi : \mathfrak{h} \times \mathfrak{g} \rightarrow \mathfrak{g}$ satisfying the following conditions:
    
\begin{itemize}
    \item[-] $\rho$ defines a representation of the Lie algebra $\mathfrak{g}$ on the vector space $\mathfrak{h}$ (i.e. $\rho_{[x, y]_\mathfrak{g}} h = \rho_x \rho_y h - \rho_y \rho_x h$),

    \item[-] $\psi$ defines a representation of the Lie algebra $\mathfrak{h}$ on the vector space $\mathfrak{g}$ (i.e. $\psi_{[h, k]_\mathfrak{h}} x = \psi_h \psi_k x - \psi_k \psi_h x$),

    \item[-] the maps $\rho$ and $\psi$ satisfy the compatibilities:
    \begin{align}
        \rho_x ([h, k]_\mathfrak{h}) =~& [\rho_x h , k]_\mathfrak{h} + [h, \rho_x k]_\mathfrak{h} + \rho_{\psi_k x} h - \rho_{\psi_h x} k, \label{11}\\
        \psi_h ([x, y]_\mathfrak{g}) =~& [\psi_h x , y]_\mathfrak{g} + [x, \psi_u y]_\mathfrak{g} + \psi_{\rho_y h} x - \psi_{\rho_x h} y, \label{22}
    \end{align}
    \end{itemize}
    for $x, y \in \mathfrak{g}$ and $h, k \in \mathfrak{h}$.
\end{defn}

\begin{exam}\label{lierep-exam}
Let $\mathfrak{g}$ be a Lie algebra and $(V, \rho)$ be a representation of it. If we equip $V$ with the trivial Lie bracket then the quadruple $(\mathfrak{g}, V, \rho, 0)$ is a matched pair of Lie algebras.
\end{exam}

\begin{exam}
    Let $\mathfrak{g}, \mathfrak{h}$ be two Lie algebras and $\mathfrak{g}$ acts on the Lie algebra $\mathfrak{h}$ by derivations. That is, there is a Lie algebra homomorphism $\rho : \mathfrak{g} \rightarrow \mathrm{Der} (\mathfrak{h})$. Then it is easy to see that the quadruple $(\mathfrak{g}, \mathfrak{h}, \rho, 0)$ is a matched pair of Lie algebras.
\end{exam}

\begin{thm}\label{thm-bicross}
    Let $(\mathfrak{g}, \mathfrak{h}, \rho, \psi)$ be a matched pair of Lie algebras. Then the direct sum $\mathfrak{g} \oplus \mathfrak{h}$ inherits a Lie algebra structure with the bracket
    \begin{align}\label{bicross}
        [ (x, h), (y, k)]_\Join := \big( [x, y]_\mathfrak{g} + \psi_h y - \psi_k x , [h, k]_\mathfrak{h} + \rho_x k - \rho_y h   \big), \text{ for } (x, h) , (y, k) \in \mathfrak{g} \oplus \mathfrak{h}.
    \end{align}
\end{thm}

The Lie algebra constructed in the above theorem is called the {\em bicrossed product} and it is denoted by $\mathfrak{g} \Join \mathfrak{h}$.

\begin{remark}
    Let $\mathfrak{g}$ and $\mathfrak{h}$ be two Lie algebras. Suppose there is a Lie algebra structure on the direct sum $\mathfrak{g} \oplus \mathfrak{h}$ for which $\mathfrak{g}$ and $\mathfrak{h}$ are both Lie subalgebras. Then there exist Lie algebra representations $\rho : \mathfrak{g} \times \mathfrak{h} \rightarrow \mathfrak{h}$ and $\psi : \mathfrak{h} \times \mathfrak{g} \rightarrow \mathfrak{g}$ that makes the quadruple $(\mathfrak{g}, \mathfrak{h}, \rho, \psi)$ into a matched pair of Lie algebras. Explicitly, the maps $\rho$ and $\psi$ are respectively given by
    \begin{align*}
        \rho_x h := \text{pr}_2 [(x, 0), (0, h)] \quad \text{ and } \quad \psi_h x := \text{pr}_1 [(0, h), (x, 0)], \text{ for } x \in \mathfrak{g}, h \in \mathfrak{h}.
    \end{align*}
    Here $\text{pr}_1: \mathfrak{g} \oplus \mathfrak{h} \rightarrow \mathfrak{g}$ and $\text{pr}_2: \mathfrak{g} \oplus \mathfrak{h} \rightarrow \mathfrak{h}$ are the projections. Moreover, the Lie bracket on $\mathfrak{g} \oplus \mathfrak{h}$ can be written by the formula (\ref{bicross}).
\end{remark}

Let $(\mathfrak{g}, \mathfrak{h}, \rho, \psi)$ and $(\mathfrak{g}', \mathfrak{h}', \rho', \psi')$ be two matched pairs of Lie algebras.  A morphism from $(\mathfrak{g}, \mathfrak{h}, \rho, \psi)$ to $(\mathfrak{g}', \mathfrak{h}', \rho', \psi')$ is given by a pair $(f, g)$ of Lie algebra homomorphisms $f: \mathfrak{g} \rightarrow \mathfrak{g}'$ and $g : \mathfrak{h} \rightarrow \mathfrak{h}'$ satisfying
\begin{align}\label{mpl-mor}
    g (\rho_x h) = \rho'_{f(x)} g(h)    ~~~~ \text{ and } ~~~~  f ( \psi_h x) = \psi'_{g(h)} f(x), \text{ for } x \in \mathfrak{g}, h \in \mathfrak{h}.
\end{align}

It follows from (\ref{mpl-mor}) that the pair $(f,g)$ defines a morphism between matched pairs of Lie algebras if and only if the map
\begin{align*}
    f \Join g : \mathfrak{g} \Join \mathfrak{h} \rightarrow \mathfrak{g}' \Join \mathfrak{h}' ~~~ \text{ defined by } ~~~ (f \Join g)((x, h)) := (f(x), g(h)), \text{ for } (x, h) \in \mathfrak{g} \Join \mathfrak{h}
\end{align*}
is a homomorphism of Lie algebras. Thus, a morphism $(f, g)$ between matched pairs of Lie algebras may also be understood by the Lie algebra homomorphism $f \Join g: \mathfrak{g} \Join \mathfrak{h} \rightarrow \mathfrak{g}' \Join \mathfrak{h}'$ if no confusion arises.

\section{Representations of a matched pair of Lie algebras}\label{sec3}

In this section, we introduce
representations of a matched pair of Lie algebras and give some examples. Finally, we define the semidirect product in the context of matched pairs of Lie algebras.

\begin{defn}\label{defn-mpl-rep}
    Let $(\mathfrak{g}, \mathfrak{h}, \rho, \psi)$ be a matched pair of Lie algebras. A {\em representation} of $(\mathfrak{g}, \mathfrak{h}, \rho, \psi)$ is given by a quadruple $(V, W, \alpha, \beta)$ in which

    \begin{itemize}
        \item $V$ is a representation for both the Lie algebras $\mathfrak{g}$ and $\mathfrak{h}$ (by the action maps $\rho_V : \mathfrak{g} \times V \rightarrow V$ and $\psi_V : \mathfrak{h} \times V \rightarrow V$, respectively),

        \item $W$ is a representation for both the Lie algebras $\mathfrak{g}$ and $\mathfrak{h}$ (by the action maps $\rho_W : \mathfrak{g} \times W \rightarrow W$ and $\psi_W : \mathfrak{h} \times W \rightarrow W$, respectively),

        \item $\alpha: V \times \mathfrak{h} \rightarrow W$ and $\beta: W \times \mathfrak{g} \rightarrow V$ are bilinear maps (called the pairing maps) satisfying the following set of identities:
        \begin{align}
            \alpha_{ (\rho_V)_x v} h =~& (\rho_W)_x \alpha_v h - \alpha_v \rho_x h, \label{1-iden}\\
            \beta_{ (\psi_W)_h w} x =~& (\psi_V)_h \beta_w x - \beta_w \psi_h x, \label{2-iden}\\
            (\rho_W)_x (\psi_W)_h w =~& (\psi_W)_{\rho_x h} w + (\psi_W)_h (\rho_W)_x w + \alpha_{\beta_w x} h - (\rho_W)_{\psi_h x} w, \label{3-iden}\\
            \alpha_v [h, k]_\mathfrak{h} =~& - (\psi_W)_k \alpha_v h + (\psi_W)_h \alpha_v k + \alpha_{ (\psi_V)_k v} h - \alpha_{ (\psi_V)_h v} k, \label{4-iden}\\ 
            (\psi_V)_h (\rho_V)_x v =~& (\rho_V)_{\psi_h x} + (\rho_V)_x (\psi_V)_h v + \beta_{\alpha_v h} x - (\psi_V)_{\rho_x h} v, \label{5-iden}\\
            \beta_w [x, y]_\mathfrak{g} =~& - (\rho_V)_y \beta_w x + (\rho_V)_x \beta_w y + \beta_{ (\rho_W)_y w} x - \beta_{ (\rho_W)_x w } y, \label{6-iden}
        \end{align}
    \end{itemize}
    for $x, y \in \mathfrak{g}$, $h, k \in \mathfrak{h}$, $v \in V$ and $w \in W$.
\end{defn}

\begin{exam} (Adjoint representation)
    Let $(\mathfrak{g}, \mathfrak{h}, \rho, \psi)$ be a matched pair of Lie algebras. Then the same quadruple $(\mathfrak{g}, \mathfrak{h}, \rho, \psi)$ can be realized as a representation, where the Lie algebra $\mathfrak{g}$ acts on the space $\mathfrak{g}$ (resp. $\mathfrak{h}$) by the adjoint action (resp. by the map $\rho$) and the Lie algebra $\mathfrak{h}$ acts on the space $\mathfrak{g}$  (resp. $\mathfrak{h}$) by the map $\psi$ (resp. by the adjoint action). We call this the adjoint representation.
\end{exam}

\begin{exam}(Representation induced by a morphism)
Let $(\mathfrak{g}, \mathfrak{h}, \rho, \psi)$ and $(\mathfrak{g}', \mathfrak{h}', \rho', \psi')$ be two matched pairs of Lie algebras, and $(f, g)$ be a morphism between them. Then the quadruple $(\mathfrak{g}', \mathfrak{h}', \alpha, \beta)$ is a representation of $(\mathfrak{g}, \mathfrak{h}, \rho, \psi)$, where the Lie algebra $\mathfrak{g}$ (resp. $\mathfrak{h}$) acts on the vector spaces $\mathfrak{g}'$ and $\mathfrak{h}'$ by the map $f$ (resp. $g$). The pairing maps $\alpha : \mathfrak{g}' \times \mathfrak{h} \rightarrow \mathfrak{h}'$ and $\beta : \mathfrak{h}' \times \mathfrak{g} \rightarrow \mathfrak{g}'$ are respectively given by
\begin{align*}
    \alpha_{x'} h := \rho'_{x'} g(h), ~~~ \beta_{h'} x := \psi'_{h'} f(x), \text{ for } x' \in \mathfrak{g}', h \in \mathfrak{h}, h' \in \mathfrak{h}', x \in \mathfrak{g}.
\end{align*}
\end{exam}

\begin{exam}(Representations of a LieRep pair) Let $\mathfrak{g}$ be a Lie algebra and $(V, \rho)$ be a representation of it.  In \cite{arnal,kuper} the authors called the pair $(\mathfrak{g}, (V, \rho))$ of a Lie algebra $\mathfrak{g}$ with a representation $(V, \rho)$ by the name of a LieRep pair. It follows from Example \ref{lierep-exam} that a LieRep pair $(\mathfrak{g}, (V, \rho))$ can be regarded as a matched pair of Lie algebras $(\mathfrak{g}, V, \rho, 0)$. In \cite{jiang} the authors have considered the representation of a LieRep pair. Precisely, a representation of a LieRep pair $(\mathfrak{g}, (V, \rho))$ is given by a triple $(\mathcal{V}, \mathcal{W}, \alpha)$ in which $\mathcal{V}, \mathcal{W}$ are both representations of the Lie algebra $\mathfrak{g}$ with the action maps $\rho_\mathcal{V}, \rho_\mathcal{W}$, and $\alpha : \mathcal{V} \times V \rightarrow \mathcal{W}$ is a bilinear map satisfying
\begin{align*}
    \alpha_{ (\rho_\mathcal{V})_x { \overline{u}}} v = (\rho_\mathcal{W})_x \alpha_{ \overline{u}} v - \alpha_{ \overline{u}} \rho_x v,
\end{align*}
for all $x \in \mathfrak{g}$, ${ \overline{u}} \in \mathcal{V}$ and $v \in V$. It follows that $(\mathcal{V}, \mathcal{W}, \alpha)$ is a representation of the LieRep pair $(\mathfrak{g}, (V, \rho))$ if and only if the quadruple $(\mathcal{V}, \mathcal{W}, \alpha, \beta = 0)$ is a representation of the matched pair of Lie algebras $(\mathfrak{g}, V, \rho, 0)$.
    \end{exam}



In the following result, we show that a matched pair of Lie algebras and a representation gives rise to a new matched pair of Lie algebras.

\begin{thm}\label{semid-thm}
    (Semidirect product) Let $(\mathfrak{g}, \mathfrak{h}, \rho, \psi)$ be a matched pair of Lie algebras and $(V, W, \alpha, \beta)$ be a representation of it. Then the quadruple $(\mathfrak{g} \ltimes V, \mathfrak{h} \ltimes W, \rho \ltimes \alpha, \psi \ltimes \beta)$ is a matched pair of Lie algebras, where the maps
    \begin{align*}
        \rho \ltimes \alpha : (\mathfrak{g} \oplus V) \times (\mathfrak{h} \oplus W) \rightarrow \mathfrak{h} \oplus W ~~~~ \text{ and } ~~~~ \psi \ltimes \beta : (\mathfrak{h} \oplus W) \times (\mathfrak{g} \oplus V) \rightarrow \mathfrak{g} \oplus V
    \end{align*}
    are respectively given by
    \begin{align*}
        (\rho \ltimes \alpha)_{(x, v)} (h, w) := \big( \rho_x h , (\rho_W)_x w + \alpha_v h \big) ~~~~ \text{ and } ~~~~ (\psi \ltimes \beta)_{(h, w)} (x, u) := \big(  \psi_h x , (\psi_V)_h u + \beta_w x\big).
    \end{align*}
\end{thm}

\begin{proof}
    Note that $\mathfrak{g} \ltimes V$ and $\mathfrak{h} \ltimes W$ are both Lie algebras with the Lie brackets (both denoted by the same notation) are respectively given by
    \begin{align*}
        [(x, u), (y, v)] := \big( [x, y]_\mathfrak{g}, (\rho_V)_x v - (\rho_V)_y u \big) ~~~ \text{ and } ~~~ [(h, w), (k, w')] := \big( [h, k]_\mathfrak{h}, (\psi_W)_h w' - (\psi_W)_k w \big),
    \end{align*}
    for $(x, u), (y, v) \in \mathfrak{g} \oplus V$ and $(h, w) , (k, w') \in \mathfrak{h} \oplus W$. We first claim that the map $\rho \ltimes \alpha$ defines a representation of the Lie algebra $\mathfrak{g} \ltimes V$ on the space $\mathfrak{h} \oplus W$. To show this, we observe that
    \begin{align*}
        &(\rho \ltimes \alpha)_{(x, u)} (\rho \ltimes \alpha)_{(y, v)} (h, w) - (\rho \ltimes \alpha)_{(y, v)} (\rho \ltimes \alpha)_{(x, u)} (h, w) \\
        &= (\rho \ltimes \alpha)_{(x, u)} \big( \rho_y h, (\rho_W)_y w + \alpha_v h   \big)    - (\rho \ltimes \alpha)_{(y, v)} \big(  \rho_x h , (\rho_W)_x w + \alpha_u h  \big) \\
        &= \big( \rho_x \rho_y h , (\rho_W)_x (\rho_W)_y w + (\rho_W)_x \alpha_v h + \alpha_u \rho_y h  \big) - \big(  \rho_y \rho_x h , (\rho_W)_y (\rho_W)_x w + (\rho_W)_y \alpha_u h - \alpha_v \rho_x h \big) \\
        &= \big(  \rho_{[x, y]_\mathfrak{g}} h , (\rho_W)_{[x,y]_\mathfrak{g}} w + \alpha_{(\rho_V)_x v} h - \alpha_{(\rho_V)_y u} h \big) \quad (\text{using } (\ref{1-iden}))\\
        &= (\rho \ltimes \alpha)_{( [x, y]_\mathfrak{g}, (\rho_V)_x v - (\rho_V)_y u )} (h, w) \\
        &= (\rho \ltimes \alpha)_{[(x, u), (y, v)]} (h, w),
    \end{align*}
    for $(x, u), (y, v) \in \mathfrak{g} \oplus V$ and $(h, w) \in \mathfrak{h} \oplus W$. This proves our claim. Similarly, by using the identity (\ref{2-iden}), one can prove that
    \begin{align*}
        (\psi \ltimes \beta)_{(h, w)}  (\psi \ltimes \beta)_{(k, w')} (x, u) -  (\psi \ltimes \beta)_{(k, w')}  (\psi \ltimes \beta)_{(h, w)} (x, u) =  (\psi \ltimes \beta)_{[(h, w), (k, w')]} (x, u),
    \end{align*}
    for $(h, w), (k,w') \in \mathfrak{h} \oplus W$ and $(x, u) \in \mathfrak{g} \oplus V$. This shows that the map $\psi \ltimes \beta$ defines a representation of the Lie algebra $\mathfrak{h} \ltimes W$ on the space $\mathfrak{g} \oplus V$. Thus, we are only left to show the compatibility conditions. For any $(x, u) \in \mathfrak{g} \oplus V$ and $(h, w) , (k, w') \in \mathfrak{h} \oplus W$, we have
    \begin{align*}
       & [ (\rho \ltimes \alpha)_{(x, u)} (h, w), (k, w')] + [(h, w), (\rho \ltimes \alpha)_{(x, u)} (k, w')] \\
       & + (\rho \ltimes \alpha)_{ (\psi \ltimes \beta)_{(k, w')} (x, u)   } (h, w) - (\rho \ltimes \alpha)_{ (\psi \ltimes \beta)_{(h, w)} (x, u)   } (k, w') \\
       & = [ ( \rho_x h, (\rho_W)_x w + \alpha_u h  ) , (k, w')] + [ (h, w), ( \rho_x k, (\rho_W)_x w' + \alpha_u k) ] \\
       & \quad + (\rho \ltimes \alpha)_{ (\psi_k x , (\psi_V)_k u + \beta_{w'} x)} (h, w) - (\rho \ltimes \alpha)_{  (\psi_h x , (\psi_V)_h u + \beta_w x)} (k, w') \\
        &= \big( [\rho_x h, k]_\mathfrak{h} , (\psi_W)_{\rho_x h} w' - (\psi_W)_k (\psi_W)_x w - \psi_k \alpha_u h  \big) \\
        &  \quad  + \big(  [h, \rho_x k]_\mathfrak{h} , (\psi_W)_h (\rho_W)_x w' + (\psi_W) \alpha_u k - (\psi_W)_{\rho_x k} w \big) 
        + \big( \rho_{\psi_k x} h, (\rho_W)_{\psi_k x} w + \alpha_{(\psi_V)_k u} h + \alpha_{\beta_{w'} x} h  \big)  \\ &  \quad  - \big( \rho_{\psi_h x} k, (\rho_W)_{\psi_h x} w' + \alpha_{ (\psi_V)_h u} k + \alpha_{\beta_w x} k   \big) \\
       & = \big(  \rho_x [h, k]_\mathfrak{h}, (\rho_W)_x \psi_h w' - (\rho_W)_x \psi_k w + \alpha_u [h, k]_\mathfrak{h}  \big) \quad (\text{using } (\ref{11}), (\ref{3-iden}), (\ref{4-iden}))\\
        &= (\rho \ltimes \alpha)_{(x, u)} ([h, k]_\mathfrak{h} , \psi_h w' - \psi_k w) \\
       & = (\rho \ltimes \alpha)_{(x, u)} [(h, w), (k, w')].
    \end{align*}
    This proves the first compatibility condition. Similarly, by using the identities (\ref{22}), (\ref{5-iden}), (\ref{6-iden}), one can prove the second compatibility condition, namely,
    \begin{align*}
    [(\psi \ltimes \beta)_{(h, w)} (x, u) , &~(y, v)] + [(x, u), (\psi \ltimes \beta)_{(h, w)} (y, v)] + (\psi \ltimes \beta)_{(\rho \ltimes \alpha)_{(y, v)} (h, w)} (x, u) \\
    &-  (\psi \ltimes \beta)_{(\rho \ltimes \alpha)_{(x, u)} (h, w)} (y, v) = (\psi \ltimes \beta)_{(h,w)} [(x, u), (y, v)],
    \end{align*}
    for $(h, w) \in \mathfrak{h} \oplus W$ and $(x, u), (y, v) \in \mathfrak{g} \oplus V$. This completes the proof that $(\mathfrak{g} \ltimes V, \mathfrak{h} \ltimes W, \rho \ltimes \alpha, \psi \ltimes \beta)$ is a matched pair of Lie algebras.
\end{proof}

The next result shows that a representation of a matched pair of Lie algebras induces a representation of the corresponding bicrossed product Lie algebra.

\begin{thm}\label{vw-thm}
    Let $(\mathfrak{g}, \mathfrak{h}, \rho, \psi)$ be a matched pair of Lie algebras and $(V, W, \alpha, \beta)$ be a representation of it. Then the vector space $V \oplus W$ can be given a representation of the bicrossed product Lie algebra $\mathfrak{g} \Join \mathfrak{h}$ with the action map $\rho_\Join : (\mathfrak{g} \Join \mathfrak{h}) \times (V \oplus W) \rightarrow V \oplus W$ given by
    \begin{align*}
        (\rho_\Join)_{(x, h)} (v, w) := \big( (\rho_V)_x v + (\psi_V)_h v - \beta_w x , (\psi_W)_h w + (\rho_W)_x w - \alpha_v h \big).
    \end{align*}
\end{thm}

\begin{proof}
Since $(V, W, \alpha, \beta)$ is a representation of the matched pair of Lie algebras $(\mathfrak{g}, \mathfrak{h}, \rho, \psi)$, by Theorem \ref{semid-thm}, one may construct the corresponding semidirect product matched pair of Lie algebras 
\begin{align*}
    (\mathfrak{g} \ltimes V, \mathfrak{h} \ltimes W, \rho \ltimes \alpha, \psi \ltimes \beta).
\end{align*}
The corresponding bicrossed product Lie algebra $(\mathfrak{g} \ltimes V) \Join (\mathfrak{h} \ltimes W)$ is given on the direct sum $\mathfrak{g} \oplus V \oplus \mathfrak{h} \oplus W$ and the bracket is given by
\begin{align*}
    \llbracket (x, u, h, w) , (y, v, k, w') \rrbracket_\Join  :=~& \big(  [(x, u), (y, v)] + (\psi \ltimes \beta)_{(h, w)} (y, v) - (\psi \ltimes \beta)_{(k, w')} (x, u) , \\
    & \qquad [(h, w), (k, w')] + (\rho \ltimes \alpha)_{(x, u)} (k, w') -  (\rho \ltimes \alpha)_{(y, v)} (h, w) \big) \\
    & \in (\mathfrak{g} \oplus V) \oplus (\mathfrak{h} \oplus W).
\end{align*}
By the identification $(\mathfrak{g} \oplus V) \oplus (\mathfrak{h} \oplus W) \longleftrightarrow (\mathfrak{g} \oplus \mathfrak{h}) \oplus (V \oplus W)$, one can observe that the above bracket can be equivalently written as
\begin{align*}
    \big(  [(x, h), (y, k)]_\Join, (\rho_\Join)_{(x, h)} (v, w') - (\rho_\Join)_{(y, k)} (u, w)   \big).
\end{align*}
It follows from the above expression that $\rho_\Join$ defines a representation of the Lie algebra $\mathfrak{g} \Join \mathfrak{h}$ on the vector space $V \oplus W$.
\end{proof}

\begin{remark}\label{remark-rep}
Let $(\mathfrak{g}, \mathfrak{h}, \rho, \psi)$ be a matched pair of Lie algebras. Consider the bicrossed product Lie algebra $\mathfrak{g} \Join \mathfrak{h}$. Let $V$ and $W$ be two vector spaces such that the direct sum $V \oplus W$ carries a representation of the Lie algebra $\mathfrak{g} \Join \mathfrak{h}$ (with the action map $\rho_\Join$) for which the following conditions are hold:

- $\rho_\Join$ restrict to representations of both the Lie algebras $\mathfrak{g}$ and $\mathfrak{h}$ on the vector space $V$,

- $\rho_\Join$ restrict to representations of both the Lie algebras $\mathfrak{g}$ and $\mathfrak{h}$ on the vector space $W$.

\noindent Then we may define bilinear maps $\alpha : V \times \mathfrak{h} \rightarrow W$ and $\beta : W \times \mathfrak{g} \rightarrow V$ by
\begin{align*}
    \alpha_v h := -\mathrm{pr}_2 \big(  (\rho_\Join)_{(0, h)} (v, 0)  \big) ~~~~ \text{ and } ~~~~ \beta_w x := - \mathrm{pr}_1 \big(  (\rho_\Join)_{(x, 0)} (0, w)   \big),
\end{align*}
for $v \in V$, $w \in W$, $x \in \mathfrak{g}$ and $h \in \mathfrak{h}$. Here $\mathrm{pr}_1 : V \oplus W \rightarrow V$ and $\mathrm{pr}_2 : V \oplus W \rightarrow W$ are the projections. Then $(V, W, \alpha , \beta)$ becomes a representation of the matched pair of Lie algebras $(\mathfrak{g}, \mathfrak{h}, \rho, \psi)$.
\end{remark}

\begin{remark}
In \cite{lu-weins} the authors have introduced the notion of a double Lie group to study dressing transformations in a Poisson-Lie group. A double Lie group is also known as a matched pair of Lie groups \cite{etingof, majid}. More precisely, a matched pair of Lie groups is a quadruple $(G, H, \varphi, \psi)$ in which $G, H$ are Lie groups and $\varphi: G \times H \rightarrow H$, $\psi: H \times G \rightarrow G$ are smooth Lie group actions satisfying additionally
\begin{align*}
    \varphi_g (h h') =  (\varphi_g h) ( \varphi_{ \psi_h g} h')  ~~ \text{ and } ~~ \psi_h (g g') = ( \psi_{\varphi_{g'} h} g) (\psi_h g'), \text{ for } g, g' \in G \text{ and } h, h' \in H.
\end{align*}
As a consequence, the cartesian product $G \times H$ inherits a Lie group structure  (called the bicrossed product, denoted by $G \Join H$) such that $G$ and $H$ are both closed Lie subgroups of $G \Join H$ and the map $G \times H \rightarrow G \Join H$ defined by $(g, h) \mapsto (g, 1_H) (1_G, h)$ is a diffeomorphism. The differentiation of a matched pair of Lie groups $(G, H, \varphi, \psi)$ yields a matched pair of Lie algebras $(\mathfrak{g}, \mathfrak{h}, \rho, \psi)$.

A representation of a matched pair of Lie groups $(G, H, \varphi, \psi)$ can be defined by a triple $(V \oplus W, V, W)$ where $V \oplus W$ carries a representation of the Lie group $G \Join H$ which restrict to representations of both the Lie groups $G$ and $H$ on both the vector spaces $V$ and $W$. Then by taking the differentiation of all Lie group representations, one obtains representations of the corresponding Lie algebras. Hence by Remark \ref{remark-rep}, we get a representation of the corresponding matched pair of Lie algebras $(\mathfrak{g}, \mathfrak{h}, \rho, \psi)$.
\end{remark}

Given a representation of a Lie algebra $\mathfrak{g}$, one can define its dual representation. More precisely, let $\mathfrak{g}$ be a Lie algebra and $V$ be a representation of it with the action map $\rho_V: \mathfrak{g} \times V \rightarrow V$. Then the dual vector space $V^*$ can also be given a representation of the Lie algebra $\mathfrak{g}$ by the action map
\begin{align*}
    \rho_{V^*} : \mathfrak{g} \times V^* \rightarrow V^* ~~~~ \text{ defined by } ~~ ((\rho_{V^*})_x q) (v) = - q ((\rho_V)_x v), \text{ for } x \in \mathfrak{g}, q \in V^*, v \in V.
\end{align*}

In the following result, we show that the above dual construction can be generalized to the context of matched pairs of Lie algebras.

\begin{prop}
    (Dual representation) Let $(\mathfrak{g}, \mathfrak{h}, \rho, \psi)$ be a matched pair of Lie algebras and $(V, W, \alpha, \beta)$ be a representation of it. Then $(W^*, V^*, \alpha^*, \beta^*)$ is also a representation of $(\mathfrak{g}, \mathfrak{h}, \rho, \psi)$, where the pairing maps $\alpha^* : W^* \times \mathfrak{h} \rightarrow V^*$ and $\beta^* : V^* \times \mathfrak{g} \rightarrow W^*$ are respectively given by
    \begin{align*}
        (\alpha^*_p h) (v) = - p (\alpha_v h) ~~~~ \text{ and } ~~~~ (\beta^*_q x) (w) = - q (\beta_w x), \text{ for } p \in W^*, h \in \mathfrak{h}, q \in V^*, x \in \mathfrak{g}.
    \end{align*}
\end{prop}

\begin{proof}
    Since $(V, W, \alpha, \beta)$ is a representation of the matched pair of Lie algebras $(\mathfrak{g}, \mathfrak{h}, \rho, \psi)$, it follows that the direct sum $V \oplus W$ can be given a representation of the bicrossed product Lie algebra $\mathfrak{g} \Join \mathfrak{h}$ with the action map $\rho_\Join$ (see Theorem \ref{vw-thm}). By dualizing this representation, we have that the vector space $(V \oplus W)^* = V^* \oplus W^*$ can be given a representation of the Lie algebra $\mathfrak{g} \Join \mathfrak{h}$. By identifying $V^* \oplus W^* \cong W^* \oplus V^*$, this dual representation is simply given by
    \begin{align*}
        (\rho^*_\Join)_{(x, h)} (p, q) = \big(  (\rho_{W^*})_x p + (\psi_{W^*})_h p - \beta^*_q x , (\psi_{V^*})_h q + (\rho_{V^*})_x q - \alpha^*_p h  \big),
    \end{align*}
    for $(x, h) \in \mathfrak{g} \Join \mathfrak{h}$ and $(p, q) \in W^* \oplus V^*$. This expression precisely shows that $(W^*, V^*, \alpha^*, \beta^*)$ is a representation of the matched pair of Lie algebras $(\mathfrak{g}, \mathfrak{h}, \rho, \psi)$.
\end{proof}

\begin{exam}(Coadjoint representation) 
   Let $(\mathfrak{g}, \mathfrak{h}, \rho, \psi)$ be a matched pair of Lie algebras. Then $(\mathfrak{h}^*, \mathfrak{g}^*, \alpha, \beta)$ is a representation, where the pairing maps $\alpha : \mathfrak{h}^* \times \mathfrak{h} \rightarrow \mathfrak{g}^*$ and $\beta : \mathfrak{g}^* \times \mathfrak{g} \rightarrow \mathfrak{h}^*$ are respectively given by
   \begin{align*}
       (\alpha_p h) (x) := - p (\rho_x h) ~~~ \text{ and } ~~~ (\beta_q x) (h) := - q (\psi_h x), \text{ for } p \in \mathfrak{h}^*, h \in \mathfrak{h}, q \in \mathfrak{g}^*, x \in \mathfrak{g}.
   \end{align*}
\end{exam}

\section{Cohomology of a matched pair of Lie algebras}\label{sec4}
In this section, we aim to define the cohomology of a matched pair of Lie algebras. Given two vector spaces $\mathfrak{g}$ and $\mathfrak{h}$ (not necessarily equipped with any additional structures), we first construct a graded Lie algebra whose Maurer-Cartan elements correspond to matched pairs of Lie algebra structures on the pair of vector spaces $(\mathfrak{g}, \mathfrak{h})$. This characterization allows us to define the cohomology of a given matched pair of Lie algebras with coefficients in the adjoint representation. We show that there is a morphism from the cohomology of a matched pair of Lie algebras to the Chevalley-Eilenberg cohomology of the bicrossed product Lie algebra. Moreover, there is a morphism from the cohomology of a Lie bialgebra to the cohomology of the corresponding matched pair of Lie algebras. Finally, we define the cohomology of a matched pair of Lie algebras with coefficients in an arbitrary representation. 

\subsection{Maurer-Cartan characterization and cohomology of a matched pair of Lie algebras}

Let $\mathfrak{g}$ and $\mathfrak{h}$ be two vector spaces. Consider the graded Lie algebra
\begin{align*}
    \mathcal{G} = \big( \oplus_{n \geq 0} \mathrm{Hom} (\wedge^{n+1} (\mathfrak{g} \oplus \mathfrak{h}), \mathfrak{g} \oplus \mathfrak{h}), [~,~]_\mathrm{NR} \big)
\end{align*}
on the space of all skew-symmetric multilinear maps on the vector space $\mathfrak{g} \oplus \mathfrak{h}$ with the Nijenhuis-Richardson bracket. Note that there is an isomorphism $\wedge^{n+1} (\mathfrak{g} \oplus \mathfrak{h}) \cong \oplus_{\substack{k+l = n+1\\k, l \geq 0}} (\wedge^k \mathfrak{g} \otimes \wedge^l \mathfrak{h})$.

A map $f \in  \mathrm{Hom} (\wedge^{n+1} (\mathfrak{g} \oplus \mathfrak{h}), \mathfrak{g} \oplus \mathfrak{h})$ is said to have {\em bidegree} $k|l$ with $k+l = n+1$ and $k, l \geq -1$ if it satisfies
\begin{align*}
    f (\wedge^{k+1} \mathfrak{g} \otimes \wedge^l  \mathfrak{h}) \subset \mathfrak{g}, \quad f (\wedge^k  \mathfrak{g} \otimes \wedge^{l+1}  \mathfrak{h}) \subset \mathfrak{h} ~~~ \text{ and } ~~ f = 0 \text{ otherwise}.
\end{align*}
We denote the set of all maps of bidegree $k|l$ by $C^{k|l} ( \mathfrak{g} \oplus  \mathfrak{h};  \mathfrak{g} \oplus  \mathfrak{h})$ or simply by $C^{k|l}$. It follows that
\begin{align*}
    C^{k|l} \cong \mathrm{Hom} (\wedge^{k+1}  \mathfrak{g} \otimes \wedge^l \mathfrak{h},  \mathfrak{g}) \oplus \mathrm{Hom} (\wedge^{k}  \mathfrak{g} \otimes \wedge^{l+1} \mathfrak{h},  \mathfrak{h}).
\end{align*}
Moreover, we have $\mathcal{G}^n = \mathrm{Hom} (\wedge^{n+1} (  \mathfrak{g} \oplus \mathfrak{h}),  \mathfrak{g} \oplus \mathfrak{h} ) \cong C^{n+1|-1} \oplus C^{n|0} \oplus \cdots \oplus C^{0|n} \oplus C^{-1|n+1}$.

\begin{lemma}\label{lemma-g}
    Let $f \in C^{k_f | l_f}$ and $g \in C^{k_g | l_g}$. Then we have
    \begin{align*}
        [f, g]_\mathrm{NR} \in C^{k_f + k_g | l_f + l_g}.
    \end{align*}
    It follows that $\oplus_{n \geq 0} C^{n|0}$ and $\oplus_{n \geq 0} C^{0|n}$ are both graded Lie subalgebras of $\mathcal{G}$.
\end{lemma}

In the following, we construct another graded Lie subalgebra of $\mathcal{G}$ that plays a fundamental role in the study of matched pairs of Lie algebras. For each $n \geq 0$, we define
\begin{align*}
    \mathcal{M}^n := C^{n|0} \oplus C^{n-1|1} \oplus \cdots \oplus C^{1 | n-1} \oplus C^{0 | n} \subset \mathcal{G}^n = \mathrm{Hom} ( \wedge^{n+1} (  \mathfrak{g} \oplus \mathfrak{h}),  \mathfrak{g} \oplus \mathfrak{h} ).
\end{align*}

\begin{prop}\label{prop-gla}
    With the above notations, the graded space $\oplus_{n \geq 0} \mathcal{M}^n$ is a graded Lie subalgebra of $\mathcal{G}$.
\end{prop}

\begin{proof}
    Let $f \in \mathcal{M}^m$ and $g \in \mathcal{M}^n$. If $f \in C^{i|m-i}$ and $g \in C^{j|n-j}$ (with $0 \leq i \leq m$ and $0\leq j \leq n$) then by Lemma \ref{lemma-g}, we have $[f, g]_\mathrm{NR} \in C^{i+j | m+n -i - j} \subset \mathcal{M}^{m+n}$. Hence the result follows.
\end{proof}

Next, let $\mathfrak{g}$ and $\mathfrak{h}$ be two vector spaces equipped with four bilinear maps 
\begin{align}\label{four-bi}
    \mu = [~,~]_\mathfrak{g} : \mathfrak{g} \times \mathfrak{g} \rightarrow \mathfrak{g}, \quad \nu = [~,~]_\mathfrak{h} : \mathfrak{h} \times \mathfrak{h} \rightarrow \mathfrak{h}, \quad \rho : \mathfrak{g} \times \mathfrak{h} \rightarrow \mathfrak{h} ~~~ \text{ and } ~~~ \psi: \mathfrak{h} \times \mathfrak{g} \rightarrow \mathfrak{g}
\end{align}
in which $\mu$ and $\nu$ are skew-symmetric. We define two elements $\mu \ltimes \rho \in C^{1|0}$ and $\psi \rtimes \nu \in C^{0|1}$ by
\begin{align*}
    (\mu \ltimes \rho) ((x, h) , (y, k)) := \big(  [x, y]_\mathfrak{g}, \rho_x k - \rho_y h  \big) ~~~ \text{ and } ~~~ (\psi \rtimes \nu) ((x, h), (y, k)) := \big( \psi_h y - \psi_k x  , [h, k]_\mathfrak{h}  \big),
\end{align*}
for $(x, h), (y, k) \in \mathfrak{g} \oplus \mathfrak{h}$. Hence we obtain an element
\begin{align*}
    \pi= ( \mu \ltimes \rho , \psi \rtimes \nu ) \in C^{1|0} \oplus C^{0|1} = \mathcal{M}^1.
\end{align*}
Conversely, any element of $\mathcal{M}^1$ corresponds to four bilinear maps like (\ref{four-bi}) in which $\mu$ and $\nu$ are skew-symmetric.

\begin{thm}\label{mc-mpl}
    With the above notations, the quadruple $( (\mathfrak{g}, \mu), (\mathfrak{h}, \nu), \rho, \psi)$ forms a matched pair of Lie algebras if and only if $\pi = (  \mu \ltimes \rho, \psi \rtimes \nu ) \in \mathcal{M}^1$ is a Maurer-Cartan element of the graded Lie algebra $( \oplus_{n \geq 0} \mathcal{M}^n, [~,~]_\mathrm{NR}).$
\end{thm}

\begin{proof}
    We have
    \begin{align*}
        &[\pi, \pi]_\mathrm{NR} \\
        &= [ (  \mu \ltimes \rho , \psi \rtimes \nu), (  \mu \ltimes \rho , \psi \rtimes \nu ) ]_\mathrm{NR} \\
        &= \big( [ \mu \ltimes \rho,  \mu \ltimes \rho]_\mathrm{NR},  [\mu \ltimes \rho, \psi \rtimes \nu]_\mathrm{NR} +      [ \psi \rtimes \nu ,  \mu \ltimes \rho]_\mathrm{NR}, [\psi \rtimes \nu, \psi \rtimes \nu]_\mathrm{NR}   \big) \\
         &= \big( [ \mu \ltimes \rho,  \mu \ltimes \rho]_\mathrm{NR},  2 [\mu \ltimes \rho, \psi \rtimes \nu]_\mathrm{NR} , [\psi \rtimes \nu, \psi \rtimes \nu]_\mathrm{NR}   \big).
    \end{align*}
    Thus, $\pi$ is a Maurer-Cartan element of the graded Lie algebra $( \oplus_{n \geq 0} \mathcal{M}^n , [~,~]_\mathrm{NR})$ if and only if
    \begin{align*}
         [ \mu \ltimes \rho,  \mu \ltimes \rho]_\mathrm{NR} = 0, \quad  [\mu \ltimes \rho, \psi \rtimes \nu]_\mathrm{NR} = 0  ~~~ \text{ and } ~~~  [\psi \rtimes \nu, \psi \rtimes \nu]_\mathrm{NR}  = 0.
    \end{align*}
    It has been observed in \cite{arnal} that $[ \mu \ltimes \rho,  \mu \ltimes \rho]_\mathrm{NR} = 0$ if and only if $\mu = [~,~]_\mathfrak{g}$ defines a Lie algebra structure on $\mathfrak{g}$ and $\rho$ defines a representation of the Lie algebra $(\mathfrak{g}, [~,~]_\mathfrak{g})$ on the vector space $\mathfrak{h}$. Similarly, $[\psi \rtimes \nu, \psi \rtimes \nu]_\mathrm{NR}  = 0$ if and only if $\nu = [~,~]_\mathfrak{h}$ defines a Lie algebra structure on $\mathfrak{h}$ and $\psi$ defines a representation of the Lie algebra $(\mathfrak{h}, [~,~]_\mathfrak{h})$ on the vector space $\mathfrak{g}$. Finally, we observe that
    \begin{align*}
      &[\mu \ltimes \rho,  \psi \rtimes \nu]_\mathrm{NR}  ((x, h), (y, k), (z, k')) \\
      &= \big(  i_{(\mu \ltimes \rho)}  (\psi \rtimes \nu) + i_{(\psi \rtimes \nu)}  (\mu \ltimes \rho)  \big)  ((x, h), (y, k), (z, k'))  \\
&=  (\psi \rtimes \nu) \big(   (\mu \ltimes \rho) ( (x, h), (y, k)) , (z, k')     \big) 
       - (\psi \rtimes \nu) \big(   (\mu \ltimes \rho) ( (x, h), (z, k')) , (y, k)     \big) \\
    & \quad  + (\psi \rtimes \nu) \big(   (\mu \ltimes \rho) ( (y, k), (z, k')) , (x, h)     \big) + (\mu \ltimes \rho) \big(   (\psi \rtimes \nu) ( (x, h), (y, k)) , (z, k')     \big) \\
    & \quad - (\mu \ltimes \rho) \big(   (\psi \rtimes \nu) ( (x, h), (z, k')) , (y, k)     \big)
      + (\mu \ltimes \rho) \big(   (\psi \rtimes \nu) ( (y, k), (z, k')) , (x, h)     \big)  \\  
       &=  (\psi \rtimes \nu) \big(  ( [x, y]_\mathfrak{g}, \rho_x k - \rho_y h), (z, k')  \big) 
      - (\psi \rtimes \nu) \big(  ( [x, z]_\mathfrak{g}, \rho_x k' - \rho_z h), (y, k)  \big) \\ 
      & \quad + (\psi \rtimes \nu) \big(  ( [y, z]_\mathfrak{g}, \rho_y k' - \rho_z k), (x, h)  \big) 
        + (\mu \ltimes \rho) \big(   (  \psi_h y - \psi_k x, [h, k]_\mathfrak{h}), (z, k')  \big) \\
        & \quad - (\mu \ltimes \rho) \big(  ( \psi_h z - \psi_{k'} x, [h, k']_\mathfrak{h}) , (y, k) \big) + (\mu \ltimes \rho) \big(  ( \psi_k z - \psi_{k'} y, [k, k']_\mathfrak{h}), (x, h)  \big) \\
       &=  \big(  \psi_{\rho_x k} z - \psi_{\rho_y h} z - \psi_{k'} [x, y]_\mathfrak{g} , [\rho_x k, k']_\mathfrak{h} - [\rho_y h, k']_\mathfrak{h}  \big) \\
     &  \quad - \big(    \psi_{\rho_x k'} y - \psi_{\rho_z h} y - \psi_{k} [x, z]_\mathfrak{g} , [\rho_x k', k]_\mathfrak{h} - [\rho_z h, k]_\mathfrak{h}  \big) \\
     &  \quad + \big(   \psi_{\rho_y k'} x - \psi_{\rho_z k} x - \psi_{h} [y, z]_\mathfrak{g} , [\rho_y k', h]_\mathfrak{h} - [\rho_z k, h]_\mathfrak{h}  \big) \\
    &  \quad + \big(  [\psi_h y, z]_\mathfrak{g} - [\psi_k x, z]_\mathfrak{g}, \rho_{\psi_h y} k' - \rho_{\psi_k x} k' - \rho_z [h, k]_\mathfrak{h}  \big) \\
      &  \quad - \big(  [\psi_h z, y]_\mathfrak{g} - [\psi_{k'} x, y]_\mathfrak{g}, \rho_{\psi_h z} k - \rho_{\psi_{k'} x} k - \rho_y [h, k']_\mathfrak{h}  \big) \\
     &  \quad  + \big(  [\psi_k z, x]_\mathfrak{g} - [\psi_{k'} y, x]_\mathfrak{g}, \rho_{\psi_k z} h - \rho_{\psi_{k'} y} h - \rho_x [k, k']_\mathfrak{h}  \big),
    \end{align*}
    for $(x, h) , (y, k), (z, k') \in \mathfrak{g} \oplus \mathfrak{h}$. It follows from the above expression that $[\mu \ltimes \rho , \psi \rtimes \nu]_\mathrm{NR} = 0$ if and only if the maps $\rho$ and $\psi$ satisfy the two compatibility conditions (\ref{11}), (\ref{22}) of a matched pair of Lie algebras. Hence the result follows.
\end{proof}

The above result gives rise to the Maurer-Cartan characterization of a matched pair of Lie algebras. This characterization enables us to construct the cochain complex associated with a matched pair of Lie algebras. Let $(\mathfrak{g}, \mathfrak{h}, \rho, \psi)$ be a matched pair of Lie algebras. For each $n \geq 0$, we define the $n$-th cochain group as
\begin{align*}
    C^0 (\mathfrak{g}, \mathfrak{h}, \rho, \psi) := \mathfrak{g} \oplus \mathfrak{h} ~~~~ \text{ and } ~~~~ C^{n \geq 1} (\mathfrak{g}, \mathfrak{h}, \rho, \psi) := \mathcal{M}^{n-1} = C^{n-1|0} \oplus C^{n-2|1} \oplus \cdots \oplus C^{1|n-2} \oplus C^{0|n-1}.
\end{align*}
Thus, an element $F \in C^{n \geq 1} (\mathfrak{g} , \mathfrak{h}, \rho, \psi) = C^{n-1|0} \oplus \cdots \oplus C^{0|n-1}$ is given by an $n$-tuple
$F = (F_1, F_2 \ldots, F_n),$ where $F_r \in C^{n-r| r-1} = \mathrm{Hom} (\wedge^{n-r+1} \mathfrak{g} \otimes \wedge^{r-1} \mathfrak{h} , \mathfrak{g}) \oplus \mathrm{Hom} (\wedge^{n-r} \mathfrak{g} \otimes \wedge^{r} \mathfrak{h} , \mathfrak{h})$ for $r=1, \ldots, n$.
There is a map $\delta_\mathrm{MPL} : C^n (\mathfrak{g}, \mathfrak{h}, \rho, \psi) \rightarrow C^{n+1} (\mathfrak{g}, \mathfrak{h}, \rho, \psi)$ given by 
\begin{align}
\delta_\mathrm{MPL}((x, h))&(y, k) := - [(x,h), (y, k)]_\Join, \text{ for } (x, h), (y, k) \in \mathfrak{g} \oplus \mathfrak{h}, \nonumber \\
  &  \delta_\mathrm{MPL} (F) := - [\pi, F], \text{ for } F \in C^n (\mathfrak{g} , \mathfrak{h}, \rho, \psi).\label{cob}
\end{align}
With the notation $F = (F_1, F_2 \ldots, F_n)$, the coboundary map (\ref{cob}) is given by
\begin{align}\label{exp-diff}
    \delta_\mathrm{MPL} (F) =~& - [(\mu \ltimes \rho,  \psi \rtimes \nu), (F_1, F_2 \ldots, F_n)]_\mathrm{NR} \nonumber \\
    =~&  \big( - [\mu \ltimes \rho, F_1]_\mathrm{NR}, \ldots, \underbrace{ - [\mu \ltimes \rho, F_r]_\mathrm{NR} - [\psi \rtimes \nu, F_{r-1}]_\mathrm{NR} }_{\in C^{n+1-r| r-1} ~~ (\text{at } r\text{-th place})}, \ldots, - [ \psi \rtimes \nu, F_n]_\mathrm{NR}  \big). 
\end{align}
We will give the explicit description of each components of (\ref{exp-diff}) in the next subsection when we discuss cohomology with coefficients in arbitrary representation.
Note that $\pi = (\mu \ltimes \rho, \psi \rtimes \nu) \in \mathcal{M}^1$ is a Maurer-Cartan element of the graded Lie algebra $(\oplus_{n \geq 0} \mathcal{M}^n, [~,~]_\mathrm{NR})$ which implies that $(\delta_\mathrm{MPL})^2 = 0$. In other words, $\{ C^\bullet (\mathfrak{g}, \mathfrak{h}, \rho, \psi), \delta_\mathrm{MPL} \}$ is a cochain complex.

\begin{defn}
    Let $(\mathfrak{g}, \mathfrak{h}, \rho, \psi)$ be a matched pair of Lie algebras. Then the cohomology groups of the cochain complex $\{ C^\bullet (\mathfrak{g}, \mathfrak{h}, \rho, \psi), \delta_\mathrm{MPL} \}$ are called the {\em cohomology} of the matched pair of Lie algebras $ (\mathfrak{g}, \mathfrak{h}, \rho, \psi)$. The corresponding cohomology groups are denoted by $H^\bullet_\mathrm{MPL} (\mathfrak{g}, \mathfrak{h}, \rho, \psi)$.
\end{defn}

Let $(\mathfrak{g}, \mathfrak{h}, \rho, \psi)$ be a matched pair of Lie algebras. Consider the  bicrossed product Lie algebra $\mathfrak{g} \Join \mathfrak{h} = (\mathfrak{g} \oplus \mathfrak{h}, [~,~]_\Join)$ given in Theorem \ref{thm-bicross}. For each $n \geq 0$, we consider the inclusion map 
\begin{align*}
\xymatrix{
   \Phi_n: C^n (\mathfrak{g}, \mathfrak{h}, \rho, \psi) \ar@{=}[d] \ar[r] & C^n (\mathfrak{g} \Join \mathfrak{h};  \mathfrak{g} \Join \mathfrak{h}) = \mathrm{Hom} (\wedge^n (\mathfrak{g} \oplus \mathfrak{h}), \mathfrak{g} \oplus \mathfrak{h}) \ar@{=}[d] \\
   C^{n-1|0} \oplus \cdots \oplus C^{0|n-1} \ar@{^{(}->}[r]  & C^{n|-1} \oplus C^{n-1|0} \oplus \cdots \oplus C^{0|n-1} \oplus C^{-1 | n}.
   }
\end{align*}
That is, $\Phi_n ((F_1, \ldots, F_n )) = (0, F_1, \ldots, F_n, 0)$, for $(F_1, \ldots, F_n) \in  C^n (\mathfrak{g}, \mathfrak{h}, \rho, \psi)$.

\begin{prop}
    The maps $\{ \Phi_n \}_{n \geq 0}$ defines a morphism of cochain complexes from $\{ C^\bullet (\mathfrak{g}, \mathfrak{h}, \rho, \psi), \delta_\mathrm{MPL} \}$ to the Chevalley-Eilenberg cochain complex $\{ C^\bullet (\mathfrak{g} \Join \mathfrak{h}; \mathfrak{g} \Join \mathfrak{h}), \delta_\mathrm{CE} \}$ of the bicrossed product Lie algebra $\mathfrak{g} \Join \mathfrak{h}$. Hence there is a morphism $H^\bullet_\mathrm{MPL} (\mathfrak{g}, \mathfrak{h}, \rho, \psi) \rightarrow H^\bullet_\mathrm{CE} (\mathfrak{g} \Join \mathfrak{h}; \mathfrak{g} \Join \mathfrak{h})$ between the corresponding cohomology groups. 
\end{prop}

\begin{proof}
    Let $F= (F_1, \ldots, F_n) \in C^n (\mathfrak{g}, \mathfrak{h}, \rho, \psi)$ be an arbitrary element, where $F_r \in C^{n-r|r-1}$ for each $r \geq 1$. Then we have
    \begin{align*}
        &(\Phi_{n+1} \circ \delta_\mathrm{MPL}) ((F_1, \ldots, F_n )) \\
        &= \Phi_{n+1} \big( - [\mu \ltimes \rho, F_1]_\mathrm{NR}, \ldots, - [\mu \ltimes \rho, F_r]_\mathrm{NR} - [\psi \rtimes \nu, F_{r-1}]_\mathrm{NR}, \ldots, - [\psi \rtimes \nu, F_n]_\mathrm{NR}   \big)\\
        &= \big( 0, - [\mu \ltimes \rho, F_1]_\mathrm{NR}, \ldots, - [\mu \ltimes \rho, F_r]_\mathrm{NR} - [\psi \rtimes \nu, F_{r-1}]_\mathrm{NR}, \ldots, - [\psi \rtimes \nu, F_n]_\mathrm{NR}, 0   \big)\\
       & = - [ (0, \mu \ltimes \rho, \psi \rtimes \nu, 0), (0, F_1, \ldots, F_n, 0)]_\mathrm{NR} \\
       & = \delta_\mathrm{CE} (0, F_1, \ldots, F_n, 0)\\
       & = (\delta_\mathrm{CE} \circ \Phi_n) ((F_1, \ldots, F_n )).
    \end{align*}
    This proves the first part. The second part is a consequence of the first part.
\end{proof}

The cohomology of a Lie bialgebra $(\mathfrak{g}, \mathfrak{g}^*)$ has been explicitly considered in \cite{ciccoli}. In the following, we find a morphism from the cohomology of a Lie bialgebra $(\mathfrak{g}, \mathfrak{g}^*)$ to the cohomology of the corresponding matched pair of Lie algebras $(\mathfrak{g}, \mathfrak{g}^*, \mathrm{ad}^*_\mathfrak{g}, \mathrm{ad}^*_{\mathfrak{g}^*}).$

Let $(\mathfrak{g}, \mathfrak{g}^*)$ be a Lie bialgebra. For each $n \geq 0$, we define the $n$-th cochain group $C^n_\mathrm{LieBi} (\mathfrak{g}, \mathfrak{g}^*)$ by
\begin{align*}
    C^0_\mathrm{LieBi} (\mathfrak{g} , \mathfrak{g}^*) = 0 ~~~ \text{ and } ~~~ C^{n \geq 1}_\mathrm{LieBi} (\mathfrak{g}, \mathfrak{g}^*) = \oplus_{\substack{ p+q = n+1 \\ p, q \geq 1}} \mathrm{Hom} (\wedge^p \mathfrak{g}, \wedge^q \mathfrak{g}).
\end{align*}
Thus, an element of $C^{n \geq 1}_\mathrm{LieBi} (\mathfrak{g}, \mathfrak{g}^*)$ is given by a tuple $(\xi_1, \ldots, xi_n)$, where $\xi_r \in \mathrm{Hom} (\wedge^{n-r+1} \mathfrak{g}, \wedge^q \mathfrak{g})$. Note that, for any $p, q \geq 1$, there are two maps $\delta_\mathfrak{g} : \mathrm{Hom}(\wedge^p  \mathfrak{g}, \wedge^q \mathfrak{g}) \rightarrow \mathrm{Hom} (\wedge^{p+1} \mathfrak{g}, \wedge^q \mathfrak{g})$ and $\delta_{\mathfrak{g}^*} : \mathrm{Hom}(\wedge^p  \mathfrak{g}, \wedge^q \mathfrak{g}) \rightarrow \mathrm{Hom} (\wedge^{p} \mathfrak{g}, \wedge^{q+1} \mathfrak{g})$ defined as follows. The map $\delta_\mathfrak{g}$ is the Chevalley-Eilenberg coboundary operator of the Lie algebra $\mathfrak{g}$ with coefficients in the adjoint representation in $\wedge^q \mathfrak{g}$. To define the map $\delta_{\mathfrak{g}^*}$, we first identify $\mathrm{Hom}(\wedge^p  \mathfrak{g}, \wedge^q \mathfrak{g})$ with the space $\mathrm{Hom}(\wedge^q  \mathfrak{g}^*, \wedge^p \mathfrak{g}^*)$. With this, $\delta_{\mathfrak{g}^*}$ is the Chevalley-Eilenberg coboundary operator of the Lie algebra $\mathfrak{g}^*$ with coefficients in the adjoint representation in $\wedge^p \mathfrak{g}^*$. Finally, we define a map $ \delta_\mathrm{LieBi} : C^n_\mathrm{LieBi} (\mathfrak{g}, \mathfrak{g}^*) \rightarrow C^{n+1}_\mathrm{LieBi} (\mathfrak{g}, \mathfrak{g}^*) $ by
\begin{align*}
   \delta_\mathrm{LieBi} (( \xi_1, \ldots, \xi_n )) = \big( \delta_\mathfrak{g} (\xi_1), \ldots, \delta_\mathfrak{g} (\xi_r) +  \delta_{\mathfrak{g}^*} (\xi_{r-1}), \ldots, \delta_{\mathfrak{g}^*} (\xi_n) \big).
\end{align*}
Then it has been shown in \cite{ciccoli} that $\{ C^\bullet_\mathrm{LieBi} (\mathfrak{g}, \mathfrak{g}^*), \delta_\mathrm{LieBi} \}$ is a cochain complex, and the corresponding cohomology groups are called the cohomology of the Lie bialgebra $(\mathfrak{g}, \mathfrak{g}^*)$.

Since any element of $\wedge^q \mathfrak{g}$ induces a linear map $\wedge^{q-1} \mathfrak{g}^* \rightarrow \mathfrak{g}$, there is a natural map $\mathrm{Hom} (\wedge^p \mathfrak{g}, \wedge^q \mathfrak{g}) \rightarrow \mathrm{Hom} (\wedge^p \mathfrak{g} \otimes \wedge^{q-1} \mathfrak{g}^*, \mathfrak{g}), ~\xi \mapsto \widetilde{\xi}$. On the other hand, by identifying the space $\mathrm{Hom} (\wedge^p \mathfrak{g}, \wedge^q \mathfrak{g})$ with the space $\mathrm{Hom} (\wedge^q \mathfrak{g}^*, \wedge^p \mathfrak{g}^*)$, one gets a map $\mathrm{Hom} (\wedge^p \mathfrak{g}, \wedge^q \mathfrak{g}) \rightarrow \mathrm{Hom} (\wedge^{p-1} \mathfrak{g} \otimes \wedge^q \mathfrak{g}^*, \mathfrak{g}^*), ~ \xi \mapsto \overline{\xi}$. Thus, one obtains a map
\begin{align*}
    \mathrm{Hom} (\wedge^p \mathfrak{g}, \wedge^q \mathfrak{g})  ~&\xrightarrow{\Psi} C^{p-1|q-1} (\mathfrak{g}, \mathfrak{g}^*, \mathrm{ad}^*_\mathfrak{g}, \mathrm{ad}^*_{\mathfrak{g}^*}) = \mathrm{Hom} (\wedge^p \mathfrak{g} \otimes \wedge^{q-1} \mathfrak{g}^*, \mathfrak{g}) \oplus \mathrm{Hom} (\wedge^{p-1} \mathfrak{g} \otimes \wedge^{q} \mathfrak{g}^*, \mathfrak{g}^*) \\
    \xi ~&\mapsto \widetilde{\xi} + \overline{\xi}.
\end{align*}
Then we have
\begin{align*}
    &(\Psi \circ \delta_\mathrm{LieBi}) ((\xi_1, \ldots, \xi_n)) \\
    &= \Psi \big(  \delta_\mathfrak{g} (\xi_1), \ldots, \delta_\mathfrak{g} (\xi_r) +  \delta_{\mathfrak{g}^*} (\xi_{r-1}), \ldots, \delta_{\mathfrak{g}^*} (\xi_n)  \big) \\
    &= \big( \widetilde{\delta_\mathfrak{g} (\xi_1)} + \overline{\delta_\mathfrak{g} (\xi_1)}, \ldots,  \widetilde{\delta_\mathfrak{g} (\xi_r)} + \widetilde{ \delta_{\mathfrak{g}^*} (\xi_{r-1})} + \overline{\delta_\mathfrak{g} (\xi_r)} +  \overline{\delta_{\mathfrak{g}^*} (\xi_{r-1})}, \ldots, \widetilde{\delta_{\mathfrak{g}^*} (\xi_n) } +  \overline{\delta_{\mathfrak{g}^*} (\xi_n) }  \big) \\
    &= - [ (\mu \ltimes \mathrm{ad}^*_\mathfrak{g}, \mathrm{ad}^*_{\mathfrak{g}^*} \rtimes \nu), (\widetilde{\xi_1} + \overline{\xi_1}, \ldots, \widetilde{\xi_n} + \overline{\xi_n}) ]_\mathrm{NR} \\
    &= \delta_\mathrm{MPL} \big(\widetilde{\xi_1} + \overline{\xi_1}, \ldots, \widetilde{\xi_n} + \overline{\xi_n}   \big)\\
    &= (\delta_\mathrm{MPL} \circ \Psi)  ((\xi_1, \ldots, \xi_n)).
\end{align*}
This shows that there is a morphism of cochain complexes from  $\{ C^\bullet_\mathrm{LieBi} (\mathfrak{g}, \mathfrak{g}^*), \delta_\mathrm{LieBi} \}$ to the complex $\{ C^\bullet (\mathfrak{g}, \mathfrak{g}^*, \mathrm{ad}^*_\mathfrak{g}, \mathrm{ad}^*_{\mathfrak{g}^*}), \delta_\mathrm{MPL} \}$. Hence it induces a morphism $H^\bullet_\mathrm{LieBi} (\mathfrak{g}, \mathfrak{g}^*) \rightarrow H^\bullet_\mathrm{MPL} (\mathfrak{g}, \mathfrak{g}^*, \mathrm{ad}^*_\mathfrak{g}, \mathrm{ad}^*_{\mathfrak{g}^*})$ between the corresponding cohomology groups.

\subsection{Cohomology with coefficients in a representation} In this subsection, we introduce the cohomology of a matched pair of Lie algebras with coefficients in a representation. Let $(\mathfrak{g}, \mathfrak{h}, \rho, \psi)$ be a matched pair of Lie algebras and $(V, W, \alpha, \beta)$ be a representation of it. For each $k, l \geq 0$, we consider
\begin{align*}
    C^{k|l} (\mathfrak{g} , \mathfrak{h} ; V , W) := \mathrm{Hom} (\wedge^{k+1} \mathfrak{g} \otimes \wedge^l \mathfrak{h} , V) \oplus \mathrm{Hom} (\wedge^{k} \mathfrak{g} \otimes \wedge^{l+1} \mathfrak{h} , W).
\end{align*}
Next, given any $n \geq 0$, we define the space of $n$-cochains $ C^n (\mathfrak{g} , \mathfrak{h}; V , W)$ by
\begin{align*}
    &C^0 (\mathfrak{g} , \mathfrak{h}; V , W) = V \oplus W ~~~~ \text{ and } \\
    &C^{n \geq 1} (\mathfrak{g} , \mathfrak{h}; V , W) = C^{n-1|0}  (\mathfrak{g} , \mathfrak{h}; V , W) \oplus \cdots \oplus  C^{0|n-1}  (\mathfrak{g} , \mathfrak{h}; V , W).
\end{align*}
Thus, an element $F \in C^{n \geq 1}(\mathfrak{g} , \mathfrak{h}; V , W)$ can be given by an $n$-tuple $F = (F_1, F_2, \ldots, F_n)$ where $F_r \in C^{n-r| r-1} (\mathfrak{g} , \mathfrak{h}; V , W)$. We also define a map $\delta_\mathrm{MPL} : C^n(\mathfrak{g} , \mathfrak{h}; V , W) \rightarrow C^{n+1} (\mathfrak{g} , \mathfrak{h}; V , W)$ by 
\begin{align*}
    \delta_\mathrm{MPL} ((v, w)) (x, h) := \big( (\rho_V)_x v + (\psi_V)_h v - \beta_w x , (\psi_W)_h w + (\rho_W)_x w - \alpha_v h \big),
\end{align*}
for $(v, w) \in V \oplus W$ and $(x, h) \in \mathfrak{g} \oplus \mathfrak{h}$, and
\begin{align}\label{exp-diff-co}
\delta_\mathrm{MPL} ((F_1, \ldots, F_n )) := \big(  \delta^{\mu \ltimes \rho} (F_1), \ldots, \underbrace{ \delta^{\mu \ltimes \rho} (F_r) + \delta^{\psi \rtimes \nu} (F_{r-1}) }_{\in C^{n-r+1| r-1} (\mathfrak{g} , \mathfrak{h}; V , W)  } , \ldots, \delta^{\psi \rtimes \nu} (F_n)  \big).
\end{align}
Here all the component functions are respectively given by
\begin{align*}
    \delta^{\mu \ltimes \rho} (F_r)& (x_1, \ldots, x_{n-r+2}, h_1, \ldots, h_{r-1}) \\
    &= \sum_{i=1}^{n-r+2} (-1)^{i+1} (\rho_V)_{x_i} F_r (x_1, \ldots, \widehat{x_i}, \ldots, x_{n-r+2}, h_1, \ldots, h_{r-1}) \\
    &+ \sum_{1 \leq i < j \leq n-r+2} (-1)^{i+j} F_r ([x_i , x_j]_\mathfrak{g}, x_1, \ldots, \widehat{x_i}, \ldots, \widehat{x_j}, \ldots, x_{n-r+2}, h_1, \ldots, h_{r-1}) \\
    &+ \sum_{\substack{1 \leq i \leq n-r+2 \\ 1 \leq j \leq r-1}} (-1)^{n-r+i} F_r (x_1, \ldots, \widehat{x_i}, \ldots, x_{n-r+2}, h_1, \ldots, \rho_{x_i} h_j , \ldots, h_{r-1}),
\end{align*}
\begin{align*}
     \delta^{\mu \ltimes \rho} (F_r) & (x_1, \ldots, x_{n-r+1}, h_1, \ldots, h_{r}) \\
    &= \sum_{i=1}^{n-r+1} (-1)^{i+1} (\rho_W)_{x_i} F_r (x_1, \ldots, \widehat{x_i}, \ldots, x_{n-r+1}, h_1, \ldots, h_{r}) \\
    &+ \sum_{i=1}^r (-1)^{n-r+i+1} \alpha \big(  F_r (x_1, \ldots, x_{n-r+1}, h_1, \ldots, \widehat{h_i}, \ldots, h_r), h_i  \big) \\
    &+ \sum_{1 \leq i < j \leq n-r+1} (-1)^{i+j} F_r ([x_i , x_j]_\mathfrak{g}, x_1, \ldots, \widehat{x_i}, \ldots, \widehat{x_j}, \ldots, x_{n-r+1}, h_1, \ldots, h_{r}) \\
    &+ \sum_{\substack{1 \leq i \leq n-r+1 \\ 1 \leq j \leq r}} (-1)^{n-r+i+1} F_r (x_1, \ldots, \widehat{x_i}, \ldots, x_{n-r+1}, h_1, \ldots, \rho_{x_i} h_j , \ldots, h_{r}),
\end{align*}
\begin{align*}
    \delta^{\psi \rtimes \nu} (F_r) &(x_1, \ldots, x_{n-r+1}, h_1,\ldots, h_{r}) \\
   & =  \sum_{i=1}^{n-r+1} (-1)^i \beta \big(  F_r (x_1, \ldots, \widehat{x_i}, \ldots, x_{n-r+1}, h_1, \ldots, h_r), x_i  \big)\\
    &+ \sum_{i=1}^r (-1)^{n-r+i} (\psi_V)_{h_i} F_r (x_1, \ldots, x_{n-r+1}, h_1, \ldots, \widehat{h_i}, \ldots, h_r)\\
    &+ \sum_{\substack{1 \leq i \leq n-r+1 \\ 1 \leq j \leq r}} (-1)^{n-r+j} F_r (x_1, \ldots, \psi_{h_j} x_i, \ldots, x_{n-r+1}, h_1, \ldots, \widehat{h_j}, \ldots, h_r)\\
    &+ \sum_{1 \leq j < k \leq r} (-1)^{n-r+1 + j+ k} F_r (x_1, \ldots, x_{n-r+1}, [h_j , h_k]_\mathfrak{h}, h_1, \ldots, \widehat{h_i}, \ldots, \widehat{h_j}, \ldots, h_r),
\end{align*}
\begin{align*}
     \delta^{\psi \rtimes \nu} (F_r) &(x_1, \ldots, x_{n-r}, h_1,\ldots, h_{r+1}) \\
    & = \sum_{i=1}^{r+1} (-1)^{n-r+i +1} (\psi_W)_{h_i} F_r (x_1, \ldots, x_{n-r}, h_1, \ldots, \widehat{h_i}, \ldots, h_{r+1}) \\
     &+ \sum_{\substack{ 1 \leq i \leq n-r \\ 1 \leq j \leq r+1}} (-1)^{n-r+j+1} F_r (x_1, \ldots, \psi_{h_j} x_i, \ldots , x_{n-r}, h_1, \ldots, \widehat{h_j}, \ldots, h_{r+1}) \\
    & + \sum_{1 \leq j , k \leq r+1} (-1)^{n-r+j + k} F_r (x_1, \ldots, x_{n-r}, [h_j, h_k]_\mathfrak{h}, h_1, \ldots, \widehat{h_j}, \ldots, \widehat{h_k}, \ldots, h_{r+1}),
\end{align*}
for all $F_r \in C^{n-r| r-1} (\mathfrak{g} , \mathfrak{h}; V , W)$ and all $x_i \in \mathfrak{g}$, $h_i \in \mathfrak{h}$. Then we have the following result.

\begin{prop}
    With the above notations, $\{ C^\bullet (\mathfrak{g} , \mathfrak{h}; V , W), \delta_\mathrm{MPL} \}$ is a cochain complex.
\end{prop}

\begin{proof}
Since $(\mathfrak{g}, \mathfrak{h}, \rho, \psi)$ is a matched pair of Lie algebras and $(V, W, \alpha, \beta)$ is a representation, it follows from Theorem \ref{semid-thm} that the quadruple $(\mathfrak{g} \ltimes V, \mathfrak{h} \ltimes W, \rho \ltimes \alpha, \psi \ltimes \beta)$ is a matched pair of Lie algebras. Thus, by the construction given in the previous subsection, we may consider the cochain complex 
\begin{align*}
    \{ C^\bullet (\mathfrak{g} \ltimes V, \mathfrak{h} \ltimes W, \rho \ltimes \alpha, \psi \ltimes \beta), \delta_\mathrm{MPL} \}.
\end{align*}
It is straightforward to verify that  $\{ C^\bullet(\mathfrak{g} , \mathfrak{h}; V , W), \delta_\mathrm{MPL} \}$ is a subcomplex of the above cochain complex. Hence the result follows.
\end{proof}

\begin{defn}
    Let $(\mathfrak{g}, \mathfrak{h}, \rho, \psi)$ be a matched pair of Lie algebras and $(V, W, \alpha, \beta)$ be a representation of it. Then the cohomology groups of the cochain complex $\{ C^\bullet (\mathfrak{g} , \mathfrak{h}; V , W), \delta_\mathrm{MPL} \}$ are called the {\em cohomology} of $(\mathfrak{g}, \mathfrak{h}, \rho, \psi)$ with coefficients in the representation $(V, W, \alpha, \beta)$. The cohomology groups are denoted by $H^\bullet_\mathrm{MPL} (\mathfrak{g} , \mathfrak{h}; V , W)$.
\end{defn}

When $(V, W, \alpha, \beta) = ( \mathfrak{g}, \mathfrak{h}, \rho, \psi)$ is the adjoint representation, we have $C^n(\mathfrak{g} , \mathfrak{h}; V , W) = C^n (\mathfrak{g}, \mathfrak{h}, \rho, \psi)$ for all $n \geq 0$. Moreover, the coboundary map (\ref{exp-diff-co}) coincides with that of (\ref{exp-diff}). Hence the corresponding cohomology groups are precisely $H^\bullet_\mathrm{MPL} (\mathfrak{g}, \mathfrak{h}, \rho, \psi)$ considered in the previous subsection.

\section{Infinitesimal deformations and abelian extensions}\label{sec5}

\subsection{Infinitesimal deformations} In this subsection, we study infinitesimal deformations of a matched pair of Lie algebras $(\mathfrak{g}, \mathfrak{h}, \rho, \psi)$. Our main result shows that the set of all equivalence classes of infinitesimal deformations is classified by the second cohomology group $H^2_\mathrm{MPL} (\mathfrak{g}, \mathfrak{h}, \rho, \psi)$.

\begin{defn}
    Let $(\mathfrak{g}, \mathfrak{h}, \rho, \psi)$ be a matched pair of Lie algebras. An {\em infinitesimal deformation} of $(\mathfrak{g}, \mathfrak{h}, \rho, \psi)$ is given by a quadruple $(\mu_1, \nu_1, \rho_1, \psi_1)$ of bilinear maps
    \begin{align*}
        \mu_1 : \mathfrak{g} \times \mathfrak{g} \rightarrow \mathfrak{g}, \quad \nu_1 : \mathfrak{h} \times \mathfrak{h} \rightarrow \mathfrak{h}, \quad \rho_1 : \mathfrak{g} \times \mathfrak{h} \rightarrow \mathfrak{h} ~~~~ \text{ and } ~~~~ \psi_1 : \mathfrak{h} \times \mathfrak{g} \rightarrow \mathfrak{g}
    \end{align*}
    in which $\mu_1, \nu_1$ are skew-symmetric that makes the quadruple $\big( (\mathfrak{g}[t]/(t^2), \mu_t), ( \mathfrak{h}[t]/(t^2), \nu_t), \rho_t, \psi_t \big)$ into a matched pair of Lie algebras over the ring ${\bf k}[t]/(t^2)$. Here the ${\bf k}[t]/(t^2)$-bilinear Lie brackets (on $\mathfrak{g}[t]/(t^2)$ and $\mathfrak{h}[t]/(t^2)$, respectively) and the ${\bf k}[t]/(t^2)$-bilinear maps $\rho_t, \psi_t$ are given by
    \begin{align*}
        \mu_t (x, y) :=~& [x, y]_\mathfrak{g} + t \mu_1 (x, y), \\
        \nu_t (h, k) :=~& [h, k]_\mathfrak{h} + t \nu_1 (h, k),\\
        (\rho_t)_x h :=~& \rho_x h + t (\rho_1)_x h,\\
        (\psi_t)_h x :=~& \psi_h x + t (\psi_1)_h x,
    \end{align*}
    for $x, y \in \mathfrak{g}$ and $h, k \in \mathfrak{h}.$
\end{defn}

Let $(\mu_1, \nu_1, \rho_1, \psi_1)$ be an infinitesimal deformation of the matched pair of Lie algebras $(\mathfrak{g}, \mathfrak{h}, \rho, \psi)$. Since $(\mathfrak{g}[t]/ (t^2), \mu_t)$ is a Lie algebra over the ring ${\bf k}[t]/(t^2)$, it turns out that the skew-symmetric bilinear map $\mu_1 : \mathfrak{g} \times \mathfrak{g} \rightarrow \mathfrak{g}$ satisfies
\begin{align}\label{inf-1-eqn}
    [ \mu_1 (x, y), z]_\mathfrak{g} + [ \mu_1 (y, z), x]_\mathfrak{g} + [ \mu_1 (z, x), y]_\mathfrak{g} + \mu_1 ( [x, y]_\mathfrak{g}, z) + \mu_1 ( [y, z]_\mathfrak{g}, x) + \mu_1 ( [z, x]_\mathfrak{g}, y) = 0,  
\end{align}
for all $x, y, z \in \mathfrak{g}$. Similarly, $(\mathfrak{h}[t]/(t^2), \nu_t)$ is a Lie algebra over the ring ${\bf k}[t]/(t^2)$ implies that the skew-symmetric bilinear map $\nu_1 : \mathfrak{h} \times \mathfrak{h} \rightarrow \mathfrak{h}$ satisfies
\begin{align}\label{inf-2-eqn}
    [ \nu_1 (h, k), k']_\mathfrak{h} + [ \nu_1 (k, k'), h]_\mathfrak{h} + [ \nu_1 (k', h), k]_\mathfrak{h} + \nu_1 ( [h, k]_\mathfrak{h}, k') + \nu_1 ( [k, k']_\mathfrak{h}, h) + \nu_1 ( [k', h]_\mathfrak{h}, k) = 0,
\end{align}
for $h, k, k' \in \mathfrak{h}$.
The map $\rho_t$ (resp. $\psi_t$) defines a representation of the Lie algebra $(\mathfrak{g}[t]/ (t^2), \mu_t)$ on the space $\mathfrak{h}[t]/ (t^2)$ (resp. a representation of the Lie algebra $(\mathfrak{h}[t]/ (t^2), \mu_t)$ on the space $\mathfrak{g}[t]/ (t^2)$). These conditions yield the following
\begin{align}
    \rho_{\mu_1 (x, y)} h + (\rho_1)_{[x, y]_\mathfrak{g}} h =~& \rho_x (\rho_1)_y h + (\rho_1)_x \rho_y h - \rho_y (\rho_1)_x h - (\rho_1)_y \rho_x h, \label{inf-3-eqn} \\
    \psi_{\nu_1 (h, k)} x + (\psi_1)_{[h, k]_\mathfrak{h}} x =~& \psi_h (\psi_1)_k x + (\psi_1)_h \psi_k x - \psi_k (\psi_1)_h x - (\rho_1)_k \rho_h x. \label{inf-4-eqn}
\end{align}
Finally, the maps $\rho_t$ and $\psi_t$ satisfies the compatibility conditions (\ref{11}), (\ref{22}) of a matched pair of Lie algebras. These conditions are respectively equivalent to 
\begin{align}
    \rho_x \nu_1 (h, k) + (\rho_1)_x [h, k]_\mathfrak{h} =~& [(\rho_1)_x h, k]_\mathfrak{h} + \nu_1 (\rho_x h, k) + [h, (\rho_1)_x h]_\mathfrak{h} + \nu_1 (h, \rho_x k)  \label{inf-5-eqn} \\
    &+ \rho_{  (\psi_1)_k x} h + (\rho_1)_{\psi_k x} h - \rho_{(\psi_1)_h x} k - (\rho_1)_{\psi_h x} k, \nonumber\\
    \psi_h \mu_1 (x, y) + (\psi_1)_h [x, y]_\mathfrak{g} =~& [(\psi_1)_h x, y]_\mathfrak{g} + \mu_1 (\psi_h x, y) + [x, (\psi_1)_h y]_\mathfrak{g} + \mu_1 (x, \psi_h y)  \label{inf-6-eqn} \\
    &+ \psi_{(\rho_1)_y h} x + (\psi_1)_{\rho_y h} x - \psi_{(\rho_1)_x h} y - (\psi_1)_{\rho_x h} y, \nonumber
\end{align}
for all $x, y \in \mathfrak{g}$ and $h, k \in \mathfrak{h}$.
To understand the identities (\ref{inf-1-eqn})-(\ref{inf-6-eqn}) in compact form, we introduce two elements $ \mu_1 \ltimes \rho_1 \in C^{1|0}$ and $ \psi_1 \rtimes \nu_1 \in C^{0|1}$ by
\begin{align*}
    (\mu_1 \ltimes \rho_1) \big(  (x, h), (y, k)  \big) :=~& \big( \mu_1 (x, y), (\rho_1)_x k - (\rho_1)_y h  \big),\\
    (\psi_1 \rtimes \nu_1) \big(  (x, h), (y, k)  \big) :=~& \big(  (\psi_1)_h y - (\psi_1)_k x, \nu_1 (h, k)  \big),
\end{align*}
for $(x, h), (y, k) \in \mathfrak{g} \oplus \mathfrak{h}$.
With the above notations, we have
\begin{align*}
  (\ref{inf-1-eqn}) , (\ref{inf-3-eqn})   ~~~~   \Longleftrightarrow &  ~~~~ [\mu \ltimes \rho,  \mu_1 \ltimes \rho_1]_\mathrm{NR} = 0,\\
    (\ref{inf-2-eqn}) , (\ref{inf-4-eqn})   ~~~~    \Longleftrightarrow & ~~~~ [\psi \rtimes \nu, \psi_1 \rtimes \nu_1 ]_\mathrm{NR} = 0,\\
       (\ref{inf-5-eqn}) , (\ref{inf-6-eqn})   ~~~~  \Longleftrightarrow & ~~~~ [\mu \ltimes \rho, \psi_1 \rtimes \nu_1]_\mathrm{NR} + [\psi \rtimes \nu , \mu_1 \ltimes \rho_1]_\mathrm{NR} = 0.
\end{align*}
Hence the identities (\ref{inf-1-eqn})-(\ref{inf-6-eqn}) is equivalent to the condition $ \delta_\mathrm{MPL} (( \mu_1 \ltimes \rho_1 ,  \psi_1 \rtimes \nu_1 )) = 0$.
Thus, a quadruple $( \mu_1, \nu_1, \rho_1, \psi_1 )$ is an infinitesimal deformation if and only if the element $(  \mu_1 \ltimes \rho_1 ,  \psi_1 \rtimes \nu_1 ) \in Z^2_\mathrm{MPL} (\mathfrak{g}, \mathfrak{h}, \rho, \psi)$ is a $2$-cocycle.

Let $(\mu_1, \nu_1, \rho_1, \psi_1)$ and $(\mu'_1, \nu'_1, \rho'_1, \psi'_1)$ be two infinitesimal deformations of a matched pair of Lie algebras $(\mathfrak{g}, \mathfrak{h}, \rho, \psi)$. They are said to be {\em equivalent} (we write $(\mu_1, \nu_1, \rho_1, \psi_1) \sim (\mu'_1, \nu'_1, \rho'_1, \psi'_1)$) if there exist linear maps $f: \mathfrak{g} \rightarrow \mathfrak{g}$ and $g: \mathfrak{h} \rightarrow \mathfrak{h}$ such that the pair $(\mathrm{id}_\mathfrak{g} + t f, \mathrm{id}_\mathfrak{h} + t g )$ defines a morphism of matched pairs of Lie algebras from
\begin{align*}
   ( (\mathfrak{g}[t]/(t^2) , \mu_t), (\mathfrak{h}[t]/(t^2), \nu_t ), \rho_t, \psi_t) \quad \text{ to } \quad ( (\mathfrak{g}[t]/(t^2), \mu_t') , (\mathfrak{h}[t]/(t^2), \nu_t'), \rho'_t, \psi'_t).
\end{align*}

Suppose $(\mu_1, \nu_1, \rho_1, \psi_1)$ and $(\mu'_1, \nu'_1, \rho'_1, \psi'_1)$ are two equivalent infinitesimal deformations. Since the map $\mathrm{id}_\mathfrak{g} + t  f : ( \mathfrak{g}[t]/ (t^2), \mu_t) \rightarrow ( \mathfrak{g}[t]/ (t^2), \mu'_t)$ is a morphism of Lie algebras over the ring ${\bf k}[t]/(t^2)$, it follows that
\begin{align}\label{inf-7-eqn}
    \mu_1 (x, y) - \mu_1' (x, y) = [x, f(y)]_\mathfrak{g} - f ([x, y]_\mathfrak{g}) + [f(x), y]_\mathfrak{g}, \text{ for all } x, y \in \mathfrak{g}.
\end{align}
Similarly, the map $\mathrm{id}_\mathfrak{h} + t  g : ( \mathfrak{h}[t]/ (t^2), \nu_t) \rightarrow ( \mathfrak{h}[t]/ (t^2), \nu'_t)$ is a morphism of Lie algebras over the ring ${\bf k}[t]/(t^2)$ implies that
\begin{align}\label{inf-8-eqn}
    \nu_1 (h, k) - \nu_1' (h, k) = [h, g(k)]_\mathfrak{h} - g ([h, k]_\mathfrak{h}) + [g(h) , k]_\mathfrak{h}, \text{ for all } h, k \in \mathfrak{h}.
\end{align}
Finally, we also have the identities $ (\mathrm{id}_\mathfrak{h} + t  g) ( (\rho_t)_x h) = (\rho_t')_{(\mathrm{id}_\mathfrak{g} + t  f)(x)} (\mathrm{id}_\mathfrak{h} + t  g)(h)$ and $ (\mathrm{id}_\mathfrak{g} + t  f) ( (\psi_t)_h x) = (\psi_t')_{(\mathrm{id}_\mathfrak{h} + t g)(h)} (\mathrm{id}_\mathfrak{g} + t  f)(x)$, for all $x \in \mathfrak{g}$ and $h \in \mathfrak{h}$. These conditions are respectively equivalent to
\begin{align}
    (\rho_1)_x h - (\rho_1')_x h =~& \rho_x g (h) - g (\rho_x h) + \rho_{f(x)}h, \label{inf-9-eqn}\\
    (\psi_1)_h x - (\psi_1')_h x =~& \psi_h f(x) - f(\psi_h x) + \psi_{g(h)} x, \label{inf-10-eqn}
\end{align}
for all $x \in \mathfrak{g}$ and $h \in \mathfrak{h}$.
The conditions (\ref{inf-7-eqn})-(\ref{inf-10-eqn}) can be equivalently written as 
\begin{align*}
    (\mu_1 \ltimes \rho_1, \psi_1 \rtimes \nu_1) - ( \mu_1' \ltimes \rho_1', \psi_1' \rtimes \nu_1') = \delta_\mathrm{MPL} ((f,g)),
\end{align*}
where we view $(f, g) \in \mathrm{Hom} (\mathfrak{g}, \mathfrak{g}) \oplus \mathrm{Hom} (\mathfrak{h}, \mathfrak{h}) = C^{0|0}$ as an element of $C^1 (\mathfrak{g}, \mathfrak{h}, \rho, \psi)$.
This implies that the $2$-cocycles $(  \mu_1 \ltimes \rho_1, \psi_1 \rtimes \nu_1  )$ and $(  \mu'_1 \ltimes \rho'_1, \psi'_1 \rtimes \nu'_1  )$ are cohomologous. Hence we obtain the following result.

\begin{thm}
    Let $(\mathfrak{g}, \mathfrak{h}, \rho, \psi)$ be a matched pair of Lie algebras. Then there is a bijection
    \begin{align*}
        \{ \text{the set of all infinitesimal deformations of } (\mathfrak{g}, \mathfrak{h}, \rho, \psi)\}/\sim ~ ~ \longleftrightarrow ~ H^2_\mathrm{MPL} (\mathfrak{g}, \mathfrak{h}, \rho, \psi).
    \end{align*}
\end{thm}

\medskip

\subsection{Abelian extensions} In this subsection, we study abelian extensions of a given matched pair of Lie algebras by a given representation. We show that the isomorphism classes of all such abelian extensions can be classified by the second cohomology group of the matched pair of Lie algebras with coefficients in the representation.

Let $(\mathfrak{g}, \mathfrak{h}, \rho, \psi)$ be a matched pair of Lie algebras and let $(V, W)$ be a pair of vector spaces. Note that the quadruple $(V, W, 0, 0)$ can be regarded as a matched pair of Lie algebras with trivial Lie algebra structures on both $V$ and $W$.

\begin{defn}
    An {\em abelian extension} of a matched pair of Lie algebras $(\mathfrak{g}, \mathfrak{h}, \rho, \psi)$ by a pair of vector spaces $(V, W)$ is a short exact sequence 
\begin{align}\label{ab-ext}
\xymatrix{
0 \ar[r] & V \Join W \ar[r]^{i_1 \Join i_2}  &  \widehat{\mathfrak{g}} \Join \widehat{\mathfrak{h}} \ar[r]^{j_1 \Join j_2} & \mathfrak{g} \Join \mathfrak{h} \ar[r] & 0
}
\end{align}
of matched pairs of Lie algebras.
\end{defn}

Thus, an abelian extension is given by a new matched pair of Lie algebras $(\widehat{ \mathfrak{g}}, \widehat{\mathfrak{h}}, \widehat{\rho}, \widehat{\psi})$ equipped with morphisms $i_1 \Join i_2: V \Join W \rightarrow \widehat{\mathfrak{g}} \Join \widehat{\mathfrak{h}}$ and $j_1 \Join j_2:  \widehat{\mathfrak{g}} \Join \widehat{\mathfrak{h}} \rightarrow \mathfrak{g} \Join \mathfrak{h}$ of matched pairs of Lie algebras such that (\ref{ab-ext}) is a short exact sequence.

Consider an abelian extension as of (\ref{ab-ext}). A {\em section} of the map $j_1 \Join j_2:  \widehat{\mathfrak{g}} \Join \widehat{\mathfrak{h}} \rightarrow \mathfrak{g} \Join \mathfrak{h}$ is given by a pair $(s_1, s_2)$ of linear maps $s_1 : \mathfrak{g} \rightarrow \widehat{\mathfrak{g}}$ and $s_2 : \mathfrak{h} \rightarrow \widehat{\mathfrak{h}}$ satisfying $j_1 \circ s_1 = \mathrm{id}_\mathfrak{g}$ and $j_2 \circ s_2 = \mathrm{id}_\mathfrak{h}$. Let $(s_1, s_2)$ be any section of the map $j_1 \Join j_2$. Then we define maps
\begin{align*}
    \rho_V : \mathfrak{g} \times V \rightarrow V, \quad \rho_W : \mathfrak{g} \times W \rightarrow W, \quad \psi_V : \mathfrak{h} \times V \rightarrow V \quad \text{ and } \quad \psi_W : \mathfrak{h} \times W \rightarrow W
\end{align*}
respectively by
\begin{align*}
    (\rho_V)_x v :=~& [s_1(x) , i_1 (v)]_{\widehat{\mathfrak{g}}},   &&(\rho_W)_x w := {\widehat{\rho}}_{s_1 (x)} i_2 (w),\\
    (\psi_V)_h v :=~&  {\widehat{\psi}}_{s_2 (h)} i_1 (v),  &&(\psi_W)_h w := [s_2 (h), i_2 (w)]_{\widehat{\mathfrak{h}}},
\end{align*}
for $x \in \mathfrak{g}$, $h \in \mathfrak{h}$, $v \in V$ and $w \in W$. It is then easy to check that $\rho_V$ (resp. $\rho_W$) defines a representation of the Lie algebra $\mathfrak{g}$ on the vector space $V$ (resp. $W$). The map $\psi_V$ (resp. $\psi_W$) defines a representation of the Lie algebra $\mathfrak{h}$ on the vector space $V$ (resp. $W$). We also define maps $\alpha : V \times\mathfrak{h} \rightarrow W$ and $\beta : W \times \mathfrak{g} \rightarrow V$ by
\begin{align*}
    \alpha_v h := \widehat{\rho}_{i_1 (v)} s_2 (h) \quad \text{ and } \quad \beta_w x := \widehat{\psi}_{i_2 (w)} s_1 (x),
\end{align*}
for $v \in V$, $h \in \mathfrak{h}$, $w \in W$ and $x \in \mathfrak{g}$. Using the properties of the matched pair of Lie algebras $(\widehat{g}, \widehat{h}, \widehat{\rho}, \widehat{\psi})$, it is easy to prove the following result.

\begin{prop}\label{rep-prop}
    With the above notations, the quadruple $(V, W, \alpha, \beta)$ is a representation of $(\mathfrak{g}, \mathfrak{h}, \rho, \psi).$
\end{prop}

Let $(s_1', s_2')$ be any other section of the map $j_1 \Join j_2$. Then we have
\begin{align*}
    s_1 (x) - s_1' (x) \in \mathrm{ker} (j_1) = \mathrm{im} (i_1) \quad \text{ and } \quad s_2 (h) - s_2'(h)  \in \mathrm{ker} (j_2) = \mathrm{im} (i_2),
\end{align*}
for all $x \in \mathfrak{g}$ and $h \in \mathfrak{h}.$ If $\rho_V', \rho_W', \psi_V', \psi_W', \alpha' , \beta'$ are the corresponding structures on $V$ and $W$ induced by the section $(s_1', s_2')$, then we have
\begin{align*}
    (\rho_V)_x v - (\rho'_V)_x v =~& [ s_1 (x) - s_1' (x) , i_1 (v)]_{\widehat{\mathfrak{g}}} = 0,\\
    (\rho_W)_x w  - (\rho'_W)_x w =~& \widehat{\rho}_{s_1 (x) - s_1' (x)} i_2 (w) = 0,\\
    (\psi_V)_h v - (\psi'_V)_h v =~& \widehat{\psi}_{ s_2 (h) - s_2' (h)} i_1 (v) = 0,\\
    (\psi_W)_h w - (\psi'_W)_h w =~& [s_2 (h) - s_2' (h), i_2 (w)]_{\widehat{\mathfrak{h}}} = 0,\\
    \alpha_v h - \alpha'_v h =~& \widehat{\rho}_{i_1 (v)} (s_2 (h) - s_2'(h)) = 0,\\
    \beta_w x - \beta'_w x =~& \widehat{\psi}_{i_2 (w)} (s_1 (x) - s_1'(x)) = 0.
\end{align*}
This shows that the representation given in Proposition \ref{rep-prop} is independent of the choice of the section.

\begin{defn}
    Two abelian extensions (two rows of the diagram (\ref{abel-diag})) are said to be {\em isomorphic} if there exists an isomorphism $f \Join g : \widehat{\mathfrak{g}} \Join \widehat{\mathfrak{h}} \rightarrow \widehat{\mathfrak{g}}' \Join \widehat{\mathfrak{h}}'$ of matched pairs of Lie algebras making the below diagram commutative
    \begin{align}\label{abel-diag}
        \xymatrix{
0 \ar[r] & V \Join W \ar@{=}[d] \ar[r]^{i_1 \Join i_2}  &  \widehat{\mathfrak{g}} \Join \widehat{\mathfrak{h}} \ar[d]^{f \Join g} \ar[r]^{j_1 \Join j_2} & \mathfrak{g} \Join \mathfrak{h} \ar[r] \ar@{=}[d] & 0 \\
0 \ar[r] & V \Join W \ar[r]_{i_1' \Join i_2'}  &  \widehat{\mathfrak{g}}' \Join \widehat{\mathfrak{h}}' \ar[r]_{j_1' \Join j_2'} & \mathfrak{g} \Join \mathfrak{h} \ar[r] & 0.
}
    \end{align}
\end{defn}

Let $(\mathfrak{g}, \mathfrak{h}, \rho, \psi)$ be a matched pair of Lie algebras and $(V, W, \alpha, \beta)$ be a given representation of it. We define $\mathrm{Ext} (\mathfrak{g}, \mathfrak{h}; V, W)$ to be the set of all isomorphism classes of abelian extensions of $(\mathfrak{g}, \mathfrak{h}, \rho, \psi)$ by the pair of vector spaces $(V, W)$ for which the induced representation is the prescribed one. Then we have the following result.

\begin{thm}
    Let $(\mathfrak{g}, \mathfrak{h}, \rho, \psi)$ be a matched pair of Lie algebras and $(V, W, \alpha, \beta)$ be a representation of it. Then
    \begin{align*}
        \mathrm{Ext} (\mathfrak{g}, \mathfrak{h}; V, W) \cong H^2_\mathrm{MPL} (\mathfrak{g}, \mathfrak{h}; V, W).
    \end{align*}
\end{thm}

\begin{proof}
    Let (\ref{ab-ext}) be an abelian extension. For any section $(s_1, s_2)$ of the map $j_1 \Join j_2$, we define maps
    \begin{align*}
        F_1 \in  C^{1|0}(\mathfrak{g}, \mathfrak{h}; V, W) ~~~ \text{ and } ~~~
        F_2  \in  C^{0|1} (\mathfrak{g}, \mathfrak{h}; V, W)
    \end{align*}
    by
    \begin{align*}
        \begin{cases}
            F_1 (x, y) := [s_1(x) , s_1 (y)]_{\widehat{\mathfrak{g}}} - s_1 ([x, y]_\mathfrak{g}),\\
            F_1 (x, h) := \widehat{\rho}_{s_1 (x)} s_2 (h) - s_2 (\rho_x h),
        \end{cases}
        \begin{cases}
            F_2 (x, h) := - \widehat{\psi}_{s_2 (h)} s_1(x) + s_1 (\psi_h x),\\
            F_2 (h, k) := [s_2(h) , s_2 (k)]_{\widehat{\mathfrak{h}}} - s_2 ([h, k]_\mathfrak{h}),
        \end{cases}
    \end{align*}
    for $x, y \in \mathfrak{g}$ and $h, k \in \mathfrak{h}$.
    Then by a straightforward calculation, it is easy to verify that the element $(F_1, F_2) \in C^2_\mathrm{MPL} (\mathfrak{g}, \mathfrak{h}; V, W)$ is a $2$-cocycle (i.e. $\delta_\mathrm{MPL} ((F_1, F_2)) = 0$). Hence the abelian extension (\ref{ab-ext}) gives rise to a cohomology class in $H^2_\mathrm{MPL} (\mathfrak{g}, \mathfrak{h}; V, W)$. Moreover, the cohomology class does not depend on the choice of the section.

    Next, we consider two equivalent abelian extensions as of (\ref{abel-diag}). If $(s_1, s_2)$ is any section of the map $j_1 \Join j_2$, then we have $j_1' \circ (f \circ s_1) = j_1 \circ s_1 = \mathrm{id}_\mathfrak{g}$ and $j_2' \circ (g \circ s_2) = j_2 \circ s_2 = \mathrm{id}_\mathfrak{h}$.
    This shows that $(f \circ s_1, g \circ s_2)$ is a section of the map $j_1' \Join j_2'$. If $(F_1' , F_2')$ is the $2$-cocycle corresponding to the second abelian extension and its section $(f \circ s_1, g \circ s_2)$ then 
    \begin{align*}
        F_1' (x, y) =~& [ (f \circ s_1)(x), (f \circ s_1)(y)]_{\widehat{\mathfrak{g}}'} - (f \circ s_1) ([x, y]_\mathfrak{g})  \\
        =~& f \big(  [s_1(x), s_1 (y)]_{\widehat{\mathfrak{g}}} - s_1 ([x, y]_\mathfrak{g}) \big) \\
        =~&  [s_1(x), s_1 (y)]_{\widehat{\mathfrak{g}}} - s_1 ([x, y]_\mathfrak{g}) \quad (\because ~ f|_V = \mathrm{id}_V) \\
        =~& F_1(x, y).
    \end{align*}
    Similarly, $F_1' (x, h) = F_1 (x, h)$,  $F_2' (x, h) = F_2 (x, h)$ and $F_2' (h, k) = F_2 (h, k)$.
    Hence we have $(F_1, F_2) = (F_1', F_2')$ and thus they corresponds to the same element in $H^2_\mathrm{MPL} (\mathfrak{g}, \mathfrak{h}; V, W)$. Thus, we obtain a well-defined map 
\begin{align*}
    \Theta_1 : \mathrm{Ext}(\mathfrak{g}, \mathfrak{h}; V, W) \rightarrow H^2_\mathrm{MPL} (\mathfrak{g}, \mathfrak{h}; V, W).
\end{align*}

To define a map on the other direction, we start with a $2$-cocycle $(F_1, F_2) \in Z^2_\mathrm{MPL} (\mathfrak{g}, \mathfrak{h}; V, W)$. Take $\widehat{\mathfrak{g}} = \mathfrak{g} \oplus V$ and $\widehat{\mathfrak{h}} = \mathfrak{h} \oplus W$. Then we define maps
\begin{align*}
    [~,~]_{\widehat{\mathfrak{g}}} : \widehat{\mathfrak{g}} \times \widehat{\mathfrak{g}} \rightarrow \widehat{\mathfrak{g}}, \quad  [~,~]_{\widehat{\mathfrak{h}}} : \widehat{\mathfrak{h}} \times \widehat{\mathfrak{h}} \rightarrow \widehat{\mathfrak{h}}, \quad \widehat{\rho} : \widehat{\mathfrak{g}} \times \widehat{\mathfrak{h}} \rightarrow \widehat{\mathfrak{h}} \quad \text{ and } \quad \widehat{\psi} : \widehat{\mathfrak{h}} \times \widehat{\mathfrak{g}} \rightarrow \widehat{\mathfrak{g}}
\end{align*}
by
\begin{align*}
    [(x, u), (y, v)]_{\widehat{\mathfrak{g}}} :=~& \big( [x, y]_\mathfrak{g}, (\rho_V)_x v - (\rho_V)_y u + F_1 (x, y)  \big), \\
    [(h, w), (k, w')]_{\widehat{\mathfrak{h}}} :=~& \big(  [h, k]_\mathfrak{h}, (\psi_W)_h w' - (\psi_W)_k w + F_2 (h, k)  \big),\\
    \widehat{\rho}_{(x, u)} (h, w) :=~& \big( \rho_x h, (\rho_W)_x w + \alpha_u h + F_1 (x, h)   \big), \\
    \widehat{\psi}_{(h, w)} (x, u) :=~& \big(  \psi_h x, (\psi_V)_h u + \beta_w x - F_2 (x, h)  \big),
\end{align*}
for $(x, u), (y, v) \in \widehat{\mathfrak{g}}$ and $(h, w) , (k, w') \in \widehat{\mathfrak{h}}$. Then it follows that $(\widehat{\mathfrak{g}}, [~,~]_{\widehat{\mathfrak{g}}} )$ and $(\widehat{\mathfrak{h}}, [~,~]_{\widehat{\mathfrak{h}}} )$ are both Lie algebras. The map $\widehat{\rho}$ (resp. $\widehat{\psi}$) defines a representation of the Lie algebra $\widehat{\mathfrak{g}}$ on the vector space $\widehat{\mathfrak{h}}$ (resp. representation of the Lie algebra $\widehat{\mathfrak{h}}$ on the vector space $\widehat{\mathfrak{g}}$). Additionally, the maps $\widehat{\rho}$ and $\widehat{\psi}$ satisfy the compatibility conditions (\ref{11}), (\ref{22}). In other words, $(\widehat{\mathfrak{g}}, \widehat{\mathfrak{h}}, \widehat{\rho}, \widehat{\psi})$ is a matched pair of Lie algebras. Moreover, 
\begin{align*}
\xymatrix{
0 \ar[r] & V \Join W \ar[r]^{i_1 \Join i_2}  &  \widehat{\mathfrak{g}} \Join \widehat{\mathfrak{h}} \ar[r]^{j_1 \Join j_2} & \mathfrak{g} \Join \mathfrak{h} \ar[r] & 0
}
\end{align*}
is an abelian extension, where the above maps are given by $i_1 (v) = (0, v)$, $i_2 (w) = (0, w),$ $j_1 (x, v) = x$ and $j_2 (h, w) = h.$

Next, let $(F_1', F_2')$ be another $2$-cocycle cohomologous to $(F_1, F_2)$, say $(F_1, F_2) - (F_1' , F_2') = \delta_\mathrm{MPL} (\theta, \vartheta)$, for some $(\theta , \vartheta) \in C^1_\mathrm{MPL} (\mathfrak{g}, \mathfrak{h}; V, W) = \mathrm{Hom} (\mathfrak{g}, V) \oplus \mathrm{Hom} (\mathfrak{h}, W)$. We define two linear maps $f : \mathfrak{g} \oplus V \rightarrow \mathfrak{g} \oplus V$ and $g : \mathfrak{h} \oplus W \rightarrow \mathfrak{h} \oplus W$ by
\begin{align*}
    f (x, v) = (x, v + \theta (x)) \quad \text{ and } \quad g (h, w) = (h, w + \vartheta (h)), 
\end{align*}
for $(x, v) \in \mathfrak{g} \oplus V$ and $(h, w) \in \mathfrak{h} \oplus W$. Then it is easy to verify that the pair of maps $(f, g)$ defines an isomorphism of the abelian extensions from $(\widehat{\mathfrak{g}}, \widehat{\mathfrak{h}}, \widehat{\rho}, \widehat{\psi})$ to $(\widehat{\mathfrak{g}}', \widehat{\mathfrak{h}}', \widehat{\rho}', \widehat{\psi}')$. Hence there is a well-defined map $\Theta_2 : H^2_\mathrm{MPL} (\mathfrak{g}, \mathfrak{h}; V, W) \rightarrow \mathrm{Ext} (\mathfrak{g}, \mathfrak{h}; V, W)$. Finally, one can check that the maps $\Theta_1$ and $\Theta_2$ are inverses to each other. This completes the proof.
\end{proof}

\section{Matched pairs of $L_\infty$-algebras}\label{sec6}

In this section, we first recall $L_\infty$-algebras (also called strongly homotopy Lie algebras) and their representations. Next, we define the notion of a matched pair of $L_\infty$-algebras. We show that a strict Rota-Baxter operator of weight $1$ on an $L_\infty$-algebra $(\mathcal{G} , \{ \mu_k \}_{k \geq 1})$ with $\mu_{k \geq 4} =0$ gives rise to a matched pair of $L_\infty$-algebras. Given a Lie algebra, the Nijenhuis-Richardson algebra and the cup-product graded Lie algebra form a matched pair of graded Lie algebras. Next, we construct a graded Lie algebra (generalizing the graded Lie algebra of Proposition \ref{prop-gla}) whose Maurer-Cartan elements correspond to matched pairs of $L_\infty$-algebras. In the end, we also construct the bicrossed product for a matched pair of $L_\infty$-algebras.

\begin{defn}
    An {\em $L_\infty$-algebra} is a pair $(\mathcal{G}, \{ \mu_k \}_{k \geq 1})$ consisting of a graded vector space $\mathcal{G} = \oplus_{i \in \mathbb{Z}} \mathcal{G}^i$ equipped with a collection $\{ \mu_k \in \mathrm{Hom} (\mathcal{G}^{\otimes k}, \mathcal{G}) \}_{k \geq 1}$ of graded linear maps with $\mathrm{deg} (\mu_k) = k-2$ (for $k \geq 1$) subject to satisfy the following conditions
\begin{itemize}
    \item each map $\mu_k$ is graded skew-symmetric in the sense that
    \begin{align*}
        \mu_k (x_1, \ldots, x_k) = (-1)^\sigma \epsilon (\sigma) \mu_k ( x_{\sigma (1)}, \ldots, x_{\sigma (k)}), \text{ for all } \sigma \in \mathbb{S}_k,
    \end{align*}
    \item for any $N \geq 1$ and homogeneous elements $x_1, \ldots, x_N \in \mathcal{G}$, the following identity holds:
    \begin{align*}
      \sum_{k +l = N+1} \sum_{\sigma \in \mathrm{Sh}(l, k-1)} (-1)^\sigma \epsilon (\sigma) (-1)^{l (k-1)} ~\mu_k \big( \mu_l (  x_{\sigma (1)}, \ldots, x_{\sigma (l)} ), x_{\sigma (l+1)}, \ldots, x_{\sigma (N)}  \big) = 0. 
    \end{align*}
\end{itemize}
\end{defn}

\begin{remark}
    Any Lie algebra can be seen as an $L_\infty$-algebra whose underlying graded vector space is concentrated in degree $0$. Any differential graded Lie algebra is also an $L_\infty$-algebra in which $\mu_k = 0$ for $k \geq 3$.
\end{remark}

Let $\mathcal{G} = \oplus_{i \in \mathbb{Z}} \mathcal{G}^i$ be a graded vector space. Consider the shifted space $\mathcal{G}[-1] = \oplus_{i \in \mathbb{Z}} (\mathcal{G}[-1])^i$, where $(\mathcal{G}[-1])^i = \mathcal{G}^{i-1}$ for each $i \in \mathbb{Z}$.
For any $n \in \mathbb{Z}$ and $k \geq 1$, we set 
\begin{align*}
\mathrm{Hom}_\mathrm{sym}^n ( \mathcal{G}[-1]^{\otimes k}, \mathcal{G}[-1] ) = \{  \theta \in \mathrm{Hom} (\mathcal{G}[-1]^{\otimes k}, \mathcal{G}[-1]) \big| ~ \substack{\theta \text{ is graded symmetric,}\\ \mathrm{deg}(\theta) = n }  \}.
\end{align*}
Then for any $k \geq 1$, there is a one-to-one correspondence between the sets
\begin{align*}
    \big\{  \mu \in \mathrm{Hom} (\mathcal{G}^{\otimes k}, \mathcal{G}) \big|& ~\substack{\mu \text{ is graded skew-symmetric,} \\ \mathrm{deg}(\mu) = k-2} \big\} \longleftrightarrow \mathrm{Hom}_\mathrm{sym}^{-1} ( \mathcal{G}[-1]^{\otimes k}, \mathcal{G}[-1] ), ~ \mu \longleftrightarrow \widetilde{\mu}.
\end{align*}
The correspondence is given by $\widetilde{\mu} = (-1)^{\frac{k(k-1)}{2}} s \circ \mu \circ (s^{-1})^{\otimes k}$, where $s : \mathcal{G} \rightarrow \mathcal{G}[-1]$ is the degree $1$ map that identifies $\mathcal{G}^{i-1}$ with $(\mathcal{G}[-1])^i$. 
We define 
\begin{align}\label{gr-gr}
\mathrm{Hom}^n_\mathrm{sym} (\mathcal{G}[-1]) := \oplus_{k \geq 1} \mathrm{Hom}_\mathrm{sym}^n ( \mathcal{G}[-1]^{\otimes k}, \mathcal{G}[-1] ).
\end{align}
Thus, an element in $\mathrm{Hom}^n_\mathrm{sym} (\mathcal{G}[-1])$ is given by a sum $\sum_{k \geq 1} \theta_k$, where $\theta_k \in \mathrm{Hom}_\mathrm{sym}^n ( \mathcal{G}[-1]^{\otimes k}, \mathcal{G}[-1] )$. For any $m, n \in \mathbb{Z}$, there is a bracket (generalizing the Nijenhuis-Richardson bracket (\ref{nr-br}))
\begin{align*}
    \llbracket ~, ~ \rrbracket_\mathrm{NR} : \mathrm{Hom}^m_\mathrm{sym} (\mathcal{G}[-1]) \times \mathrm{Hom}^n_\mathrm{sym} (\mathcal{G}[-1]) \rightarrow \mathrm{Hom}^{m+n}_\mathrm{sym} (\mathcal{G}[-1])
\end{align*}
given by
\begin{align*}
    &\llbracket \sum_{k \geq 1} \theta_k , \sum_{l \geq 1} \eta_l \rrbracket_\mathrm{NR} = \sum_{p \geq 1} \sum_{k+l = p+1} ( i_{\theta_k}  \eta_l - (-1)^{mn} i_{\eta_l}  \theta_k), ~~ \text{ where }\\
    (i_{\theta_k} \eta_l) &(x_1, \ldots, x_{k+l-1}) := \sum_{\sigma \in \mathrm{Sh} (k, l-1)}\epsilon (\sigma) \eta_l \big( \theta_k (x_{\sigma (1)}, \ldots, x_{\sigma (k)}), x_{\sigma (k+1)}, \ldots, x_{\sigma (k+l-1)}  \big),
\end{align*}
for $x_1, \ldots, x_{k+l-1} \in \mathcal{G}[-1]$. Then we have the following.

\begin{thm}\label{l-inf-mc}
    Let $\mathcal{G}$ be a graded vector space. 

    (i) Then $( \oplus_{n \in \mathbb{Z}} \mathrm{Hom}^n_\mathrm{sym} (\mathcal{G}[-1]), \llbracket ~,~ \rrbracket_\mathrm{NR}) $ is a graded Lie algebra.

    (ii) Let $\big\{  \mu_k \in \mathrm{Hom} (\mathcal{G}^{\otimes k}, \mathcal{G}) \big| ~\substack{\mu_k \mathrm{~ is~graded~ skew-symmetric,} \\ \mathrm{deg}(\mu_k) = k-2} \big\}_{k \geq 1}$ be a collection of graded skew-symmetric maps with specific degrees. Consider the corresponding collection $\big\{  \widetilde{\mu_k} \in \mathrm{Hom}^{-1}_\mathrm{sym} (\mathcal{G}[-1]^{\otimes k}, \mathcal{G}[-1]) \big\}_{k \geq 1}$. Then $(\mathcal{G}, \{ \mu_k \}_{k \geq 1})$ is an $L_\infty$-algebra if and only if $\sum_{k \geq 1} \widetilde{\mu_k} \in \mathrm{Hom}^{-1}_\mathrm{sym} (\mathcal{G}[-1])$ is a Maurer-Cartan element of the graded Lie algebra $( \oplus_{n \in \mathbb{Z}} \mathrm{Hom}^n_\mathrm{sym} (\mathcal{G}[-1]), \llbracket ~,~ \rrbracket_\mathrm{NR}) $.
\end{thm}

Let $( \mathcal{G}, \{ \mu_k \}_{k \geq 1})$ be an $L_\infty$-algebra. A {\em representation} of $( \mathcal{G}, \{ \mu_k \}_{k \geq 1})$ is a graded vector space $\mathcal{H} = \oplus_{i \in \mathbb{Z}} \mathcal{H}^i$ equipped with a collection $\{   \rho_k \in \mathrm{Hom} (\mathcal{G}^{\otimes k-1} \otimes \mathcal{H} , \mathcal{H})  \}_{k \geq 1}$ of graded linear maps with $\mathrm{deg} (\rho_k ) = k -2$ (for $k \geq 1$) that satisfy the following conditions:

- for each $k \geq 1$, the map $\rho_k$ is graded skew-symmetric on the inputs of $\mathcal{G}$, i.e.
\begin{align*}
    \rho_k (x_1, \ldots, x_{k-1}, h) = (-1)^\sigma \epsilon (\sigma) \rho_k ( x_{\sigma (1)}, \ldots, x_{\sigma (k-1)} , h), \text{ for all } \sigma \in \mathbb{S}_{k-1},
\end{align*}

- for each $N \geq 1$ and homogeneous elements $x_1, \ldots, x_{N-1} \in \mathcal{G}$, $x_N \in \mathcal{H}$,
\begin{align*}
    \sum_{k+l = N+1} &\sum_{\substack{\sigma \in \mathrm{Sh}(l, k-1) \\ \sigma (l) =N }} (-1)^\sigma \varepsilon (\sigma ) (-1)^{l (k-1) + (x_{\sigma (1)} + \cdots + x_{\sigma (l)} + l) (x_{ \sigma (l+1)} + \cdots + x_{\sigma (N)})} \\
    & \qquad \qquad \qquad \qquad \qquad  \rho_k \big(  x_{ \sigma (l+1)} , \ldots, x_{\sigma (N)}, \rho_l ( x_{\sigma (1)} , \ldots, x_{\sigma (l)}   )  \big) \\
    &+ \sum_{k+l = N+1} \sum_{\substack{\sigma \in \mathrm{Sh}(l, k-1) \\ \sigma (N) =N }} (-1)^\sigma \varepsilon (\sigma) (-1)^{l (k-1)} \rho_k \big(  \mu_l (  x_{\sigma (1)} , \ldots, x_{\sigma (l)}   ),   x_{ \sigma (l+1)} , \ldots, x_{\sigma (N)} \big) = 0.
\end{align*}

Let $\sigma \in \mathrm{Sh} (l, k-1)$ be a shuffle. Then we must have $\sigma (l) = N$ or $\sigma (N) = N$. Hence we have $\sum_{\sigma \in \mathrm{Sh} (l, k-1)} = \sum_{\substack{\sigma \in \mathrm{Sh} (l, k-1) \\ \sigma (l) = N}} + \sum_{\substack{\sigma \in \mathrm{Sh} (l, k-1) \\ \sigma (N) = N}}$. Thus, by using the skew-symmetry of the structure maps, one can observe that an $L_\infty$-algebra $(\mathcal{G}, \{ \mu_k \}_{k \geq 1})$ can be realized as a representation of itself. This is called the {\em adjoint representation}.

\begin{prop}
    Let $(\mathcal{G} , \{ \mu_k \}_{k \geq 1})$ be an $L_\infty$-algebra and $ \{ \rho_k \}_{k \geq 1}$ defines a representation of it on the graded vector space $\mathcal{H}$. Then the direct sum graded vector space $\mathcal{G} \oplus \mathcal{H}$ inherits an $L_\infty$-algebra structure with the structure maps $\{ \mu_k\ltimes \rho_k : (\mathcal{G} \oplus \mathcal{H})^{\otimes k} \rightarrow \mathcal{G} \oplus \mathcal{H} \}_{k \geq 1}$ that are given by
    \begin{align}\label{semi}
        &(\mu_k \ltimes \rho_k) \big(  (x_1, h_1), \ldots, (x_k , h_k)  \big) \nonumber \\
        & \qquad := \big(  \mu_k (x_1, \ldots, x_k) , \sum_{i=1}^k (-1)^{k-i} (-1)^{|h_i| ( |x_{i+1}| + \cdots + |x_k|)} \rho_k (x_1, \ldots, \widehat{x_i}, \ldots, x_k , h_i)  \big),
    \end{align}
    for $k \geq 1$ and $(x_1, h_1), \ldots , (x_k , h_k ) \in \mathcal{G} \oplus \mathcal{H}$.
\end{prop}

The $L_\infty$-algebra constructed in the above proposition is called the {\em semidirect product} of the $L_\infty$-algebra $(\mathcal{G}, \{ \mu_k \}_{k \geq 1})$ with the representation $\mathcal{H}$. It is often denoted by $\mathcal{G} \ltimes \mathcal{H}$ when all the structure operations are clear from the context.

Let $\mathcal{G} = \oplus_{i \in \mathbb{Z}} \mathcal{G}^i$ and $\mathcal{H} = \oplus_{i \in \mathbb{Z}} \mathcal{H}^i$ be two graded vector spaces (not necessarily equipped with any additional structures). Consider the shifted graded vector space $(\mathcal{G} \oplus \mathcal{H})[-1] = \oplus_{i \in \mathbb{Z}} ((\mathcal{G} \oplus \mathcal{H}) [-1])^i = \oplus_{i \in \mathbb{Z}} (\mathcal{G}^{i-1} \oplus \mathcal{H}^{i-1})$. For each $n \in \mathbb{Z}$, we define (similar to (\ref{gr-gr}))
\begin{align*}
    \mathrm{Hom}^n_\mathrm{sym} (  (\mathcal{G} \oplus \mathcal{H})[-1] ) := \oplus_{k \geq 1} \mathrm{Hom}^n_\mathrm{sym} \big(   (\mathcal{G} \oplus \mathcal{H})[-1]^{\otimes k} , (\mathcal{G} \oplus \mathcal{H})[-1]  \big).
\end{align*}
In Theorem \ref{l-inf-mc}, we have observed that the graded space $\oplus_{n \in \mathbb{Z}} \mathrm{Hom}^n_\mathrm{sym} (  (\mathcal{G} \oplus \mathcal{H})[-1] ) $ inherits a graded Lie bracket $\llbracket ~, ~ \rrbracket_\mathrm{NR}$ which makes it a graded Lie algebra. Moreover, Maurer-Cartan elements of this graded Lie algebra correspond to $L_\infty$-algebra structures on the graded vector space $\mathcal{G} \oplus \mathcal{H}$.

Suppose the graded vector spaces $\mathcal{G}$ and $\mathcal{H}$ are equipped with two collections of graded linear maps
\begin{align}
    &\{\mu_k \in \mathrm{Hom} (\mathcal{G}^{\otimes k}, \mathcal{G}) |~   \substack{ \mu_k \text{ is graded skew-symmetric}\\
    \text{ and  deg}(\mu_k ) = k-2 } \}_{k \geq 1} ~~~ \text{ and } \label{muk} \\
    &  \{\rho_k \in \mathrm{Hom} (\mathcal{G}^{\otimes k-1} \otimes \mathcal{H}, \mathcal{H}) |~   \substack{ \rho_k \text{ is graded skew-symmetric on the}\\
    \text{inputs of $\mathcal{G}$ and  deg}(\rho_k ) = k-2 } \}_{k \geq 1}. \label{nuk}
\end{align}
Consider the graded vector space $\mathcal{G} \oplus \mathcal{H}$ and endow it with the collection $\{ \mu_k \ltimes \rho_k : (\mathcal{G} \oplus \mathcal{H})^{\otimes k} \rightarrow \mathcal{G} \oplus \mathcal{H}   \}_{k \geq 1}$ of graded linear maps defined by (\ref{semi}). It is easy to observe that $\mu_k \ltimes \rho_k$ is skew-symmetric and $\mathrm{deg} (\mu_k \ltimes \rho_k) = k-2$, for $k \geq 1$. Then $(\mathcal{G} \oplus \mathcal{H}, \{ \mu_k \ltimes \rho_k \}_{k \geq 1})$ is an $L_\infty$-algebra if and only if $(\mathcal{G}, \{ \mu_k \}_{k \geq 1})$ is an $L_\infty$-algebra and $\{ \rho_k \}_{k \geq 1}$ defines a representation of it on the graded vector space $\mathcal{H}$. Combining this with the result of Theorem \ref{l-inf-mc}, we obtain the following result.

\begin{thm}\label{mc-l-inf-rep}
    Let $\mathcal{G}$ and $\mathcal{H}$ be two graded vector spaces. Suppose there are two collections of graded linear maps as of (\ref{muk}) and (\ref{nuk}). Then the following are equivalent:
\begin{itemize}
    \item[(i)] $(\mathcal{G}, \{ \mu_k \}_{k \geq 1})$ is an $L_\infty$-algebra and $ \{ \rho_k \}_{k \geq 1}$ defines a representation of it on the graded vector space $\mathcal{H}$.

\item[(ii)] The element $\sum_{k \geq 1} \widetilde{( \mu_k \ltimes \rho_k)}$ is a Maurer-Cartan element in the graded Lie algebra 
\begin{align*}
( \oplus_{n \in \mathbb{Z}} \mathrm{Hom}^n_\mathrm{sym}  ( (\mathcal{G} \oplus \mathcal{H})[-1]), \llbracket ~, ~ \rrbracket_\mathrm{NR}   ).
\end{align*}
\end{itemize}
\end{thm}

We are now in a position to introduce matched pairs of $L_\infty$-algebras.

\begin{defn}
    A {\em matched pair of $L_\infty$-algebras} is a quadruple $(\mathcal{G}, \mathcal{H}, \{ \rho_k \}_{k \geq 1}, \{ \psi_k \}_{k \geq 1})$ in which $\mathcal{G} = (\mathcal{G}, \{ \mu_k \}_{k \geq 1})$ and $\mathcal{H} = (\mathcal{H}, \{ \nu_k \}_{k \geq 1})$ are both $L_\infty$-algebras and
    \begin{align*}
        \{\rho_k \in \mathrm{Hom} (\mathcal{G}^{\otimes k-1} \otimes \mathcal{H}, \mathcal{H}) |~   \substack{ \rho_k \text{ is graded skew-symmetric on the}\\
    \text{inputs of $\mathcal{G}$ and  deg}(\rho_k ) = k-2 } \}_{k \geq 1} \\
     \{\psi_k \in \mathrm{Hom} (\mathcal{H}^{\otimes k-1} \otimes \mathcal{G}, \mathcal{G}) |~   \substack{ \psi_k \text{ is graded skew-symmetric on the}\\
    \text{inputs of $\mathcal{H}$ and  deg}(\psi_k ) = k-2 } \}_{k \geq 1} 
    \end{align*}
    are two collections of graded linear maps subject to satisfy the following conditions:

    (i) the collection $\{ \rho_k \}_{k \geq 1}$ defines a representation of the $L_\infty$-algebra $(\mathcal{G}, \{ \mu_k \}_{k \geq 1})$ on the graded vector space $\mathcal{H}$,

    (ii) the collection $\{ \psi_k \}_{k \geq 1}$ defines a representation of the $L_\infty$-algebra $(\mathcal{H}, \{ \nu_k \}_{k \geq 1})$ on the graded vector space $\mathcal{G}$

    (iii) the collections $\{ \rho_k \}_{k \geq 1}$ and $\{ \psi_k \}_{k \geq 1}$ additionally satisfy the following compatibilities: for any $N \geq 1$ and homogeneous elements $(x_1, h_1), \ldots , (x_N, h_N) \in \mathcal{G} \oplus \mathcal{H}$,

    \begin{align}
       & \sum_{k+l = N+1} \sum_{\sigma \in \mathrm{Sh}(l,k-1)} \bigg\{ (-1)^\sigma \varepsilon (\sigma) (-1)^{l(k-1)}  \mu_k \big(  \mu_l  (x_{\sigma (1)}, \ldots, x_{\sigma (l)}) , x_{\sigma (l+1)}, \ldots, x_{\sigma (N)}   \big) \nonumber  \\
       & + \sum_{i=1}^l  (-1)^\sigma \varepsilon (\sigma) (-1)^{l(k-1)} (-1)^{l-i} \mu_k \big(   \psi_l ( h_{\sigma (1)}, \ldots, \widehat{h_{\sigma (i)}}, \ldots, h_{\sigma (l)}, x_{\sigma (i)})  ,  x_{\sigma (l+1)}, \ldots, x_{\sigma (N)}   \big) \nonumber  \\
       & + \sum_{j=1}^{k-1}  (-1)^\sigma \varepsilon (\sigma) (-1)^{l(k-1)}  (-1)^{k-j+1} \psi_k \big(   \nu_l (h_{\sigma (1)}, \ldots, h_{\sigma (l)}), h_{\sigma (l+1)}, \ldots , \widehat{h_{\sigma (l+j)}}, \ldots, h_{\sigma (N)}, x_{\sigma (l+ j)}   \big) \nonumber  \\
       & + \sum_{j=1}^{k-1}  (-1)^\sigma \varepsilon (\sigma) (-1)^{(l-1)(k-1) + j+l -i} \psi_k \big(  \rho_l ( x_{\sigma (1)}, .., \widehat{x_{\sigma (i)}}, .. , x_{\sigma (l)}, h_{\sigma (i)}) , h_{\sigma (l+1)}, .. , \widehat{h_{\sigma (l+j)}}, .., h_{\sigma (N)}, x_{\sigma (l+ j)}  \big) \nonumber  \\
       & + (-1)^\sigma \varepsilon (\sigma) (-1)^{(l-1)(k-1)} \psi_k \big(  h_{\sigma (l+1)} , \ldots , h_{\sigma (N)},  \mu_l  (x_{\sigma (1)}, \ldots, x_{\sigma (l)})  \big)  \nonumber \\
       & + \sum_{i=1}^l (-1)^{k+l -i-1}  (-1)^\sigma \varepsilon (\sigma) (-1)^{l(k-1)}  \psi_k \big(    h_{\sigma (l+1)} , \ldots , h_{\sigma (N)}, \psi_l (h_{\sigma (1)}, \ldots, \widehat{h_{\sigma (i)}}, \ldots, h_{\sigma (l)}, x_{\sigma (i)}) \big) \bigg\} = 0,
    \end{align}
    \begin{align}
       & \sum_{k+l = N+1} \sum_{\sigma \in \mathrm{Sh}(l,k-1)} \bigg\{ (-1)^\sigma \varepsilon (\sigma) (-1)^{l(k-1)}  \nu_k \big(  \nu_l  (h_{\sigma (1)}, \ldots, h_{\sigma (l)}) , h_{\sigma (l+1)}, \ldots, h_{\sigma (N)}   \big) \nonumber  \\
       & + \sum_{i=1}^l  (-1)^\sigma \varepsilon (\sigma) (-1)^{l(k-1)} (-1)^{l-i} \nu_k \big(   \rho_l ( x_{\sigma (1)}, \ldots, \widehat{x_{\sigma (i)}}, \ldots, x_{\sigma (l)}, h_{\sigma (i)})  ,  h_{\sigma (l+1)}, \ldots, h_{\sigma (N)}   \big) \nonumber  \\
       & + \sum_{j=1}^{k-1}  (-1)^\sigma \varepsilon (\sigma) (-1)^{l(k-1)}  (-1)^{k-j+1} \rho_k \big(   \mu_l (x_{\sigma (1)}, \ldots, x_{\sigma (l)}), x_{\sigma (l+1)}, \ldots , \widehat{x_{\sigma (l+j)}}, \ldots, x_{\sigma (N)}, h_{\sigma (l+ j)}   \big) \nonumber  \\
       & + \sum_{j=1}^{k-1}  (-1)^\sigma \varepsilon (\sigma) (-1)^{(l-1)(k-1) + j+l -i} \rho_k \big(  \psi_l ( h_{\sigma (1)}, .., \widehat{h_{\sigma (i)}}, .. , h_{\sigma (l)}, x_{\sigma (i)}) , x_{\sigma (l+1)}, .. , \widehat{x_{\sigma (l+j)}}, .., x_{\sigma (N)}, h_{\sigma (l+ j)}  \big) \nonumber  \\
       & + (-1)^\sigma \varepsilon (\sigma) (-1)^{(l-1)(k-1)} \rho_k \big(  x_{\sigma (l+1)} , \ldots , x_{\sigma (N)},  \nu_l  (h_{\sigma (1)}, \ldots, h_{\sigma (l)})  \big)  \nonumber \\
       & + \sum_{i=1}^l (-1)^{k+l -i-1}  (-1)^\sigma \varepsilon (\sigma) (-1)^{l(k-1)}  \rho_k \big(    x_{\sigma (l+1)} , \ldots , x_{\sigma (N)}, \rho_l (x_{\sigma (1)}, \ldots, \widehat{x_{\sigma (i)}}, \ldots, x_{\sigma (l)}, h_{\sigma (i)}) \big) \bigg\} = 0.
    \end{align}
\end{defn}

\begin{remark}
    Any matched pair of Lie algebras can be realized as a matched pair of $L_\infty$-algebras.
\end{remark}

\begin{exam}
    (Matched pair of graded Lie algebras) Let $(\mathcal{G}, [~,~]_\mathcal{G})$ and $(\mathcal{H}, [~,~]_\mathcal{H})$ be two graded Lie algebras. Suppose there are degree $0$ graded linear maps $\rho : \mathcal{G} \times \mathcal{H} \rightarrow \mathcal{H}$ and $\psi : \mathcal{H} \times \mathcal{G} \rightarrow \mathcal{G}$ satisfying
    \begin{align}
        \rho_{[x,y]_\mathcal{G}} h =~& \rho_x \rho_y h - (-1)^{|x| |y|} \rho_y \rho_x h, \label{mpgla1}\\
        \psi_{[h, k]_\mathcal{H}} x =~& \psi_h \psi_k x - (-1)^{|h| |k|} \psi_k \psi_h x, \label{mpgla2}\\
        \rho_x ([h, k]_\mathcal{H}) =~& [\rho_x h, k]_\mathcal{H} + (-1)^{|x| |h|} [h, \rho_x k]_\mathcal{H} + (-1)^{(|x| + |h|)|k|} \rho_{\psi_k x} h -(-1)^{|x| |h|} \rho_{\psi_h x} k, \label{mpgla3}\\
        \psi_h ([x, y]_\mathcal{G}) =~& [\psi_h x, y]_\mathcal{G} + (-1)^{|h| |x|} [x, \psi_h y]_\mathcal{G} + (-1)^{(|h| + |x|)|y|} \psi_{\rho_y h} x - (-1)^{|h||x|} \psi_{\rho_x h} y, \label{mpgla4}
    \end{align}
    for $x, y \in \mathcal{G}$ and $h, k \in \mathcal{H}$. Then the quadruple $(\mathcal{G}, \mathcal{H}, \rho, \psi)$ is called a matched pair of graded Lie algebras. A matched pair of graded Lie algebras is an example of a matched pair of $L_\infty$-algebras.
\end{exam}

Let $(\mathfrak{g}, [~,~]_\mathfrak{g})$ be a Lie algebra. Then the graded vector space $\oplus_{n \geq 0} \mathrm{Hom} (\wedge^n \mathfrak{g}, \mathfrak{g})$ of all skew-symmetric multilinear maps carries a cup-product (ref. \cite{nij-ric3}) given by
\begin{align}\label{cupp}
    [P, Q]_\mathrm{cup} (x_1, \ldots, x_{m+n}) = \sum_{\sigma \in \mathrm{Sh} (m, n)} (-1)^\sigma [P (  x_{\sigma (1)}, \ldots, x_{\sigma (m)}  ), Q( x_{\sigma (m+1)}, \ldots, x_{\sigma (m+n)}  )]_\mathfrak{g},
\end{align}
for $P \in \mathrm{Hom} (\wedge^m \mathfrak{g}, \mathfrak{g})$, $Q \in \mathrm{Hom} (\wedge^n \mathfrak{g}, \mathfrak{g})$ and $x_1, \ldots, x_{m+n} \in \mathfrak{g}$. Note that an element $P \in \mathrm{Hom} (\wedge^m \mathfrak{g}, \mathfrak{g})$ can be identified with an element $\theta \otimes x \in \wedge^m \mathfrak{g}^* \otimes \mathfrak{g}$ by the relation $P(x_1, \ldots, x_m) = \theta (x_1, \ldots, x_m) x$, for any $x_1, \ldots, x_m \in \mathfrak{g}$. With this identification, the cup-product (\ref{cupp}) can be equivalently expressed as
\begin{align*}
    [\theta \otimes x, \vartheta \otimes y]_\mathrm{cup} = (\theta \wedge \vartheta) \otimes [x, y]_\mathfrak{g}, \text{ for } P = \theta \otimes x \text{ and } Q = \vartheta \otimes y.
\end{align*}
The cup-product $[~,~]_\mathrm{cup}$ makes the space $\oplus_{n \geq 0} \mathrm{Hom} (\wedge^n \mathfrak{g}, \mathfrak{g})$ into a graded Lie algebra. In other words, the pair $( \oplus_{n \geq 0} \mathrm{Hom} (\wedge^n \mathfrak{g}, \mathfrak{g}), [~,~]_\mathrm{cup}  )$ is a graded Lie algebra. For any $m, n \geq 0$, we consider the map
\begin{align}\label{rho-map}
    \rho : \mathrm{Hom} (\wedge^{m+1} \mathfrak{g}, \mathfrak{g}) \times \mathrm{Hom} (\wedge^{n} \mathfrak{g}, \mathfrak{g}) \rightarrow \mathrm{Hom} (\wedge^{m+n} \mathfrak{g}, \mathfrak{g}) ~\text{ given by }~ (P, Q) \mapsto \rho_P Q := i_P Q,  
\end{align}
for $P \in \mathrm{Hom} (\wedge^{m+1} \mathfrak{g}, \mathfrak{g})$ and $Q \in  \mathrm{Hom} (\wedge^{n} \mathfrak{g}, \mathfrak{g})$. Then we have the following result.

\begin{prop}
    Let $(\mathfrak{g}, [~,~]_\mathfrak{g})$ be a Lie algebra. Then the quadruple 
    \begin{align*}
    \big(     (\oplus_{n \geq 0} \mathrm{Hom} (\wedge^{n+1} \mathfrak{g}, \mathfrak{g}), [~,~]_\mathrm{NR}), (\oplus_{n \geq 0} \mathrm{Hom} (\wedge^{n} \mathfrak{g}, \mathfrak{g}), [~,~]_\mathrm{cup}), \rho, \psi= 0     \big)
    \end{align*}
    is a matched pair of graded Lie algebras.
\end{prop}

\begin{proof}
    For any $P \in \mathrm{Hom} (\wedge^{m+1} \mathfrak{g}, \mathfrak{g})$, $Q \in \mathrm{Hom} (\wedge^{n+1} \mathfrak{g}, \mathfrak{g})$ and $R \in \mathrm{Hom} (\wedge^{r} \mathfrak{g}, \mathfrak{g})$, we have
    \begin{align*} 
        \rho_{[P, Q]_\mathrm{NR}} R =~& i_{[P, Q]_\mathrm{NR}} R \\
        =~& i_{ (i_P Q - (-1)^{mn} i_Q P)} R \\
        =~& i_P i_Q R - (-1)^{mn} i_Q i_P R\\
        =~& \rho_P \rho_Q R - (-1)^{mn} \rho_Q \rho_P R.
    \end{align*}
    In the third equality, we have used the fact that the operation  $P \diamond Q:= i_Q P$  is a graded right pre-Lie product on $\oplus_{n \geq 0} \mathrm{Hom}(\wedge^{n+1} \mathfrak{g}, \mathfrak{g})$. This proves the condition (\ref{mpgla1}). Next, let $P = \theta \otimes x \in \wedge^{m+1} \mathfrak{g}^* \otimes \mathfrak{g}$, $Q = \vartheta \otimes y \in \wedge^n \mathfrak{g}^* \otimes \mathfrak{g}$ and $R = \eta \otimes z \in \wedge^r \mathfrak{g}^* \otimes \mathfrak{g}$. Then we have
    \begin{align*}
        \rho_P ([Q, R]_\mathrm{cup}) =~& \rho_{\theta \otimes x} (\vartheta \wedge \eta \otimes [y, z]_\mathfrak{g}) \\
        =~& \theta \wedge (i_x (\vartheta \wedge \eta)) \otimes [y, z]_\mathfrak{g} \\
        =~& (\theta \wedge (i_x \vartheta) \wedge \eta ) \otimes [y, z]_\mathfrak{g} + (-1)^n (\theta \wedge \vartheta \wedge (i_x \eta)) \otimes [y, z]_\mathfrak{g}.
    \end{align*}
    On the other hand,
    \begin{align*}
        [\rho_P Q, R]_\mathrm{cup} + (-1)^{mn} [Q, \rho_P R]_\mathrm{cup} 
        =~& [(\theta \wedge i_x \vartheta) \otimes y, \eta \otimes z]_\mathrm{cup} + (-1)^{mn} [ \vartheta \otimes y , (\theta \wedge i_x \eta) \otimes z]_\mathrm{cup} \\
        =~& ( \theta \wedge (i_x \vartheta ) \wedge \eta ) \otimes [y, z]_\mathfrak{g} + (-1)^{mn + (m+1)n} (\theta \wedge \vartheta \wedge (i_x \eta)) \otimes [y, z]_\mathfrak{g}.
    \end{align*}
    Hence $ \rho_P ([Q, R]_\mathrm{cup}) = [\rho_P Q, R]_\mathrm{cup} + (-1)^{mn} [Q, \rho_P R]_\mathrm{cup}$ which verifies the condition (\ref{mpgla3}) when $\psi = 0$. This proves the result.
\end{proof}

In the following, we construct a graded Lie algebra whose specific Maurer-Cartan elements correspond to matched pairs of $L_\infty$-algebras. Let $\mathcal{G}$ and $\mathcal{H}$ be two graded vector spaces. Consider the graded Lie algebra
\begin{align*}
    C = \big(  \oplus_{n \in \mathbb{Z}} \mathrm{Hom}^n_\mathrm{sym} ( (\mathcal{G} \oplus \mathcal{H})[-1]), \llbracket ~, ~ \rrbracket_\mathrm{NR}  \big)
\end{align*}
given in Theorem \ref{l-inf-mc}. Note that $${C}^n = \mathrm{Hom}^n_\mathrm{sym} ( (\mathcal{G} \oplus \mathcal{H})[-1]) = \oplus_{k \geq 1} \mathrm{Hom}^n_\mathrm{sym} \big(  (\mathcal{G} \oplus \mathcal{H})[-1]^{\otimes k}, (\mathcal{G} \oplus \mathcal{H})[-1]   \big), $$ 
for all $n \in \mathbb{Z}$.  A map $f \in \mathrm{Hom}^n_\mathrm{sym} \big(  (\mathcal{G} \oplus \mathcal{H})[-1]^{\otimes k}, (\mathcal{G} \oplus \mathcal{H})[-1]   \big)$ is said to have tridegree $n|r|s$ (with $r+s = k-1$) if 
\begin{align*}
    f ( \mathcal{G}[-1]^{\otimes r+1} \otimes \mathcal{H}[-1]^{\otimes s}) \subset \mathcal{G} [-1], \quad  f ( \mathcal{G}[-1]^{\otimes r} \otimes \mathcal{H}[-1]^{\otimes s+1}) \subset \mathcal{H} [-1] ~~~ \text{ and } ~~~ f = 0  ~ \text{ otherwise}.
\end{align*}
We denote the set of all maps of tridegree $n|r|s$ by ${C}^{n|r|s}$. That is,
\begin{align*}
    {C}^{n|r|s} \cong \mathrm{Hom}^n_\mathrm{sym} (\mathcal{G}[-1]^{\otimes r+1} \otimes \mathcal{H}[-1]^{\otimes s}, \mathcal{G}[-1]) \oplus \mathrm{Hom}^n_\mathrm{sym} (\mathcal{G}[-1]^{\otimes r} \otimes \mathcal{H}[-1]^{\otimes s+1}, \mathcal{H}[-1]).
\end{align*}
Then we have
\begin{align*}
    {C}^n = \oplus_{k \geq 1} \mathrm{Hom}^n_\mathrm{sym} \big(  (\mathcal{G} \oplus \mathcal{H})[-1]^{\otimes k}, (\mathcal{G} \oplus \mathcal{H})[-1]   \big)  = \oplus_{k \geq 1} \oplus_{\substack{r+s = k-1 \\ r, s \geq -1}} {C}^{n|r|s}.
\end{align*}

The following result is straightforward to verify.

\begin{lemma}\label{lem}
    Let $f \in C^{n_f|r_f|s_f}$ and $g \in C^{n_g | r_g | s_g}$. Then $\llbracket f, g \rrbracket_\mathrm{NR} \in C^{n_f + n_g | r_f + r_g | s_f + s_g}$.
\end{lemma}

We define a subspace $\mathcal{M}^n \subset C^n$ by 
\begin{align*}
\mathcal{M}^n = \oplus_{k \geq 1} \oplus_{\substack{r+s = k-1\\ r, s \geq 0}} C^{n|r|s}.
\end{align*}
Thus, an element $\Delta \in \mathcal{M}^n$ can be expressed as $\Delta = \oplus_{k \geq 1} \big(   (\Delta_1, \Delta_2, \ldots, \Delta_k) \big)$, where $\Delta_i \in C^{n| k-i| i-1}$. It follows from Lemma \ref{lem} that $\mathcal{M} = \oplus_{n \in \mathbb{Z}} \mathcal{M}^n$ is a graded Lie subalgebra of $C$.

Next, suppose that $\mathcal{G}$ and $\mathcal{H}$ are equipped with four collections of maps
\begin{align*}
&\{\mu_k \in \mathrm{Hom} (\mathcal{G}^{\otimes k}, \mathcal{G}) |~   \substack{ \mu_k \text{ is graded skew-symmetric}\\
    \text{ and  deg}(\mu_k ) = k-2 } \}_{k \geq 1},\\
   & \{\nu_k \in \mathrm{Hom} (\mathcal{H}^{\otimes k}, \mathcal{H}) |~   \substack{ \nu_k \text{ is graded skew-symmetric}\\
    \text{ and  deg}(\nu_k ) = k-2 } \}_{k \geq 1},\\
    &    \{\rho_k \in \mathrm{Hom} (\mathcal{G}^{\otimes k-1} \otimes \mathcal{H}, \mathcal{H}) |~   \substack{ \rho_k \text{ is graded skew-symmetric on the}\\
    \text{inputs of $\mathcal{G}$ and  deg}(\rho_k ) = k-2 } \}_{k \geq 1}, \\
    & \{\psi_k \in \mathrm{Hom} (\mathcal{H}^{\otimes k-1} \otimes \mathcal{G}, \mathcal{G}) |~   \substack{ \psi_k \text{ is graded skew-symmetric on the}\\
    \text{inputs of $\mathcal{H}$ and  deg}(\psi_k ) = k-2 } \}_{k \geq 1}. 
    \end{align*}
    For each $k \geq 1$, we define graded skew-symmetric maps $\mu_k \ltimes \rho_k, \psi_k \rtimes \nu_k : (\mathcal{G} \oplus \mathcal{H})^{\otimes k} \rightarrow \mathcal{G} \oplus \mathcal{H}$ by
\begin{align*}
    &(\mu_k \ltimes \rho_k) \big(  (x_1, h_1), \ldots, (x_k , h_k)  \big) \\
    & \qquad := \big(  \mu_k (x_1, \ldots, x_k) , \sum_{i=1}^k (-1)^{k-i} (-1)^{|h_i| ( |x_{i+1}| + \cdots + |x_k|)} \rho_k (x_1, \ldots, \widehat{x_i}, \ldots, x_k , h_i)  \big), \\
    & (\psi_k \rtimes \nu_k) \big(  (x_1, h_1), \ldots, (x_k , h_k)  \big) \\
    & \qquad := \big(  \sum_{i=1}^k (-1)^{k-i} (-1)^{|x_i| ( |h_{i+1}| + \cdots + |h_k|)} \psi_k (h_1, \ldots, \widehat{h_i}, \ldots, h_k, x_i) , \nu_k (h_1, \ldots , h_k) \big). 
\end{align*}
Then we have $\mathrm{deg} (\mu_k \ltimes \rho_k) = \mathrm{deg} (\psi_k \rtimes \nu_k) = k-2$. We may consider the maps 
\begin{align*}
\widetilde{ \mu_k \ltimes \rho_k} , \widetilde{\psi_k \rtimes \nu_k} \in \mathrm{Hom}^{-1}_\mathrm{sym} \big(  (\mathcal{G} \oplus \mathcal{H})[-1]^{\otimes k}, (\mathcal{G} \oplus \mathcal{H})[-1]\big).
\end{align*}
Then $\widetilde{ \mu_k \ltimes \rho_k} \in C^{-1| k-1| 0}$ and $\widetilde{\psi_k \rtimes \nu_k}  \in C^{-1|0 | k-1}$. Thus $\oplus_{k \geq 1} \big(  (\widetilde{ \mu_k \ltimes \rho_k} , 0, \ldots, 0, \widetilde{\psi_k \rtimes \nu_k}  )   \big) \in \mathcal{M}^{-1}$. 

\begin{thm}
    With the above notations, $\big(  (\mathcal{G}, \{\mu_k \}_{k \geq 1}), (\mathcal{H}, \{ \nu_k \}_{k \geq 1}), \{ \rho_k \}_{k \geq 1}, \{ \psi_k \}_{k \geq 1} \big)$ is a matched pair of $L_\infty$-algebras if and only if the element  $\oplus_{k \geq 1} \big(  (\widetilde{ \mu_k \ltimes \rho_k} , 0, \ldots, 0, \widetilde{\psi_k \rtimes \nu_k}  )   \big) \in \mathcal{M}^{-1}$ is a Maurer-Cartan element in the graded Lie algebra $\mathcal{M}.$
\end{thm}

\begin{proof}
    We have
    \begin{align*}
        &\llbracket \oplus_{k \geq 1} \big(  (\widetilde{ \mu_k \ltimes \rho_k} , 0, \ldots, 0, \widetilde{\psi_k \rtimes \nu_k}  )   \big) , \oplus_{k \geq 1} \big(  (\widetilde{ \mu_k \ltimes \rho_k} , 0, \ldots, 0, \widetilde{\psi_k \rtimes \nu_k}  )   \big)  \rrbracket_\mathrm{NR} \\
       & = \oplus_{p \geq 1} \oplus_{\substack{k+l = p+1 \\k, l \geq 1}} \llbracket  (\widetilde{ \mu_k \ltimes \rho_k} , 0, \ldots, 0, \widetilde{\psi_k \rtimes \nu_k}  ), (\widetilde{ \mu_l \ltimes \rho_l} , 0, \ldots, 0, \widetilde{\psi_l \rtimes \nu_l}  )   \rrbracket_\mathrm{NR} \\
        &= \oplus_{p \geq 1} \oplus_{\substack{k+l = p+1 \\k, l \geq 1}} \big( \llbracket \widetilde{ \mu_k \ltimes \rho_k}, \widetilde{\mu_l \ltimes \rho_l}  \rrbracket_\mathrm{NR}, 0, \ldots, 0,   2 \llbracket \widetilde{ \mu_k \ltimes \rho_k},  \widetilde{\psi_l \rtimes \nu_l} \rrbracket_\mathrm{NR}  , 0, \ldots, 0, \llbracket \widetilde{\psi_k \rtimes \nu_k}, \widetilde{\psi_l \rtimes \nu_l} \rrbracket_\mathrm{NR}  \big) \\
       & = \oplus_{p \geq 1} \big(     \llbracket  \sum_{k \geq 1} \widetilde{ \mu_k \ltimes \rho_k},  \sum_{k \geq 1} \widetilde{ \mu_k \ltimes \rho_k}  \rrbracket_\mathrm{NR} , 0, \ldots ,  \llbracket  \sum_{k \geq 1} \widetilde{ \mu_k \ltimes \rho_k},  \sum_{k \geq 1} \widetilde{ \psi_k \rtimes \nu_k}  \rrbracket_\mathrm{NR}, 0, \ldots,  \llbracket \sum_{k \geq 1} \widetilde{ \psi_k \rtimes \nu_k}, \sum_{k \geq 1} \widetilde{ \psi_k \rtimes \nu_k} \rrbracket_\mathrm{NR} \big).
    \end{align*}
    Thus, it follows that $\oplus_{k \geq 1} \big(  (\widetilde{ \mu_k \ltimes \rho_k} , 0, \ldots, 0, \widetilde{\psi_k \rtimes \nu_k}  )   \big)$ is a Maurer-Cartan element if and only if 
    \begin{align*}
         \llbracket  \sum_{k \geq 1} \widetilde{ \mu_k \ltimes \rho_k},  \sum_{k \geq 1} \widetilde{ \mu_k \ltimes \rho_k}  \rrbracket_\mathrm{NR} = 0, ~~~ \llbracket  \sum_{k \geq 1} \widetilde{ \mu_k \ltimes \rho_k},  \sum_{k \geq 1} \widetilde{ \psi_k \rtimes \nu_k}  \rrbracket_\mathrm{NR} = 0 ~~ \text{ and } ~~  \llbracket \sum_{k \geq 1} \widetilde{ \psi_k \rtimes \nu_k}, \sum_{k \geq 1} \widetilde{ \psi_k \rtimes \nu_k} \rrbracket_\mathrm{NR} = 0.
    \end{align*}
    In Theorem \ref{mc-l-inf-rep}, we have shown that $\llbracket  \sum_{k \geq 1} \widetilde{ \mu_k \ltimes \rho_k},  \sum_{k \geq 1} \widetilde{ \mu_k \ltimes \rho_k}  \rrbracket_\mathrm{NR} = 0$ if and only if $\mathcal{G} = (\mathcal{G}, \{ \mu_k \}_{k \geq 1})$ is an $L_\infty$-algebra and the maps $\{ \rho_k \}_{k \geq 1}$ makes the graded vector space $\mathcal{H}$ into a representation of the $L_\infty$-algebra $\mathcal{G}$. The same argument shows that $\llbracket \sum_{k \geq 1} \widetilde{ \psi_k \rtimes \nu_k}, \sum_{k \geq 1} \widetilde{ \psi_k \rtimes \nu_k} \rrbracket_\mathrm{NR} = 0$ if and only if $\mathcal{H} = (\mathcal{H}, \{ \nu_k \}_{k \geq 1})$ is an $L_\infty$-algebra and the maps $\{ \psi_k \}_{k \geq 1}$ makes the graded vector space $\mathcal{G}$ into a representation of the $L_\infty$-algebra $\mathcal{H}$. Finally, $\llbracket  \sum_{k \geq 1} \widetilde{ \mu_k \ltimes \rho_k},  \sum_{k \geq 1} \widetilde{ \psi_k \rtimes \nu_k}  \rrbracket_\mathrm{NR} = 0$ is equivalent to the compatibility conditions of the matched pair of $L_\infty$-algebras.
\end{proof}

The above result generalizes Theorem \ref{mc-mpl} in the homotopy context. In the following result, we generalize Theorem \ref{thm-bicross}.

\begin{thm}
    Let $\big( (\mathcal{G}, \{ \mu_k \}_{k \geq 1}) ,  (\mathcal{H}, \{ \nu_k \}_{k \geq 1}) , \{ \rho_k \}_{k \geq 1}, \{ \psi_k \}_{k \geq 1 }   \big)$ be a matched pair of $L_\infty$-algebras. Then the direct sum $\mathcal{G} \oplus \mathcal{H}$ inherits an $L_\infty$-algebra whose structure maps $\{ l_k : (\mathcal{G} \oplus \mathcal{H})^{\otimes k} \rightarrow \mathcal{G} \oplus \mathcal{H} \}_{k \geq 1}$ are given by
    \begin{align}\label{bicrossed-inf}
        &l_k \big(  (x_1, h_1), \ldots, (x_k, h_k)  \big) \\
        & \qquad:= \big(  \mu_k (x_1, \ldots, x_k) + \sum_{i=1}^k (-1)^{k-i} (-1)^{|x_i| ( |h_{i+1}| + \cdots + |h_k|)} \psi_k (h_1, \ldots, \widehat{h_i}, \ldots, h_k, x_i), \nonumber \\  & \qquad \qquad \qquad \nu_k (h_1, \ldots, h_k) + \sum_{i=1}^k (-1)^{k-i} (-1)^{|h_i| ( |x_{i+1}| + \cdots + |x_k|)} \rho_k (x_1, \ldots, \widehat{x_i}, \ldots, x_k, h_i)  \big), \nonumber
    \end{align}
    for any $k \geq 1$ and $(x_1, h_1), \ldots, (x_k, h_k) \in \mathcal{G} \oplus \mathcal{H}$.
\end{thm}

\begin{proof}
   Note that, for any $k \geq 1$, the map $l_k$ is simply given by $(\mu_k \ltimes \rho_k) + (\psi_k \rtimes \nu_k)$. Moreover,
   \begin{align*}
      \sum_{k \geq 1} \widetilde{l_k} =  \sum_{k \geq 1} \widetilde{(\mu_k \ltimes \rho_k) + (\psi_k \rtimes \nu_k)} = \sum_{k \geq 1} \widetilde{(\mu_k \ltimes \rho_k)} + \sum_{k \geq 1} \widetilde{(\psi_k \rtimes \nu_k) }.
   \end{align*}
   Hence, on the graded Lie algebra $\big(  \oplus_{n \in \mathbb{Z}} \mathrm{Hom}^n ( (\mathcal{G} \oplus \mathcal{H})[-1]),  \llbracket ~, ~ \rrbracket_\mathrm{NR} \big)$, we have
   \begin{align*}
   & \llbracket \sum_{k \geq 1} \widetilde{l_k}, \sum_{k \geq 1} \widetilde{l_k} \rrbracket_\mathrm{NR} \\
       &= \llbracket \sum_{k \geq 1} \widetilde{(\mu_k \ltimes \rho_k) + (\psi_k \rtimes \nu_k)}, \sum_{k \geq 1} \widetilde{(\mu_k \ltimes \rho_k) + (\psi_k \rtimes \nu_k)} \rrbracket_\mathrm{NR} \\
      & = \llbracket  \sum_{k \geq 1} \widetilde{(\mu_k \ltimes \rho_k)},  \sum_{k \geq 1} \widetilde{(\mu_k \ltimes \rho_k)} \rrbracket_\mathrm{NR} + 2 \llbracket  \sum_{k \geq 1} \widetilde{(\mu_k \ltimes \rho_k)}, \sum_{k \geq 1} \widetilde{(\psi_k \rtimes \nu_k) } \rrbracket_\mathrm{NR} + \llbracket \sum_{k \geq 1} \widetilde{(\psi_k \rtimes \nu_k) }, \sum_{k \geq 1} \widetilde{(\psi_k \rtimes \nu_k) } \rrbracket_\mathrm{NR} \\
      & = 0 + 0 + 0 = 0.
   \end{align*}
   Thus, it follows from Theorem \ref{l-inf-mc} that $(\mathcal{G} \oplus \mathcal{H}, \{  l_k  \}_{k \geq 1})$ is an $L_\infty$-algebra.
\end{proof}

\begin{remark}\label{remark-last}
The $L_\infty$-algebra $(\mathcal{G} \oplus \mathcal{H}, \{ l_k \}_{k \geq 1})$ constructed in the above theorem is called the {\em bicrossed product} and it is denoted by $\mathcal{G} \Join \mathcal{H}$. Note that $\mathcal{G}$ and $\mathcal{H}$ are both $L_\infty$-subalgebras of  $\mathcal{G} \Join \mathcal{H}$. Moreover, we have the following observation. Suppose $\mathcal{G}$ and $\mathcal{H}$ are two $L_\infty$-algebras and there is an $L_\infty$-algebra structure $(\mathcal{G} \oplus \mathcal{H}, \{ l_k \}_{k \geq 1})$ on the direct sum for which $\mathcal{G}$ and $\mathcal{H}$ are both $L_\infty$-subalgebras, and
\begin{align*}
    l_k \big(  (x_1, 0), \ldots , (x_{i}, 0) , (0, h_{i+1}), \ldots , (0, h_k)  \big) = 0, \text{ for } k \geq 4 \text{ and } i , k -i \geq 2.
\end{align*}
Then there is a representation $\{ \rho_k \}_{k \geq 1}$ of the $L_\infty$-algebra $\mathcal{G}$ on the graded vector space $\mathcal{H}$, and a representation $\{ \psi_k \}_{k \geq 1}$ of the $L_\infty$-algebra $\mathcal{H}$ on the graded vector space $\mathcal{G}$ given by
\begin{align*}
    &\rho_k (x_1, \ldots, x_{k-1}, h) = \mathrm{pr}_2 l_k \big(  (x_1, 0), \ldots, (x_{k-1}, 0), (0, h)  \big) \text{ and } \\
    \psi_k (h_1, \ldots, h_{k-1}, x) &= \mathrm{pr}_1 l_k \big( (0, h_1), \ldots, (0, h_{k-1}), (x, 0)  \big), \text{ for } k \geq 1 \text{ and } x, x_i \in \mathcal{G}, h, h_i \in \mathcal{H}.
\end{align*}
Moreover, $(\mathcal{G}, \mathcal{H}, \{ \rho_k \}_{k \geq 1}, \{ \psi_k \}_{k \geq 1})$ is a matched pair of $L_\infty$-algebras.
\end{remark}

In \cite{zhu} the authors have shown that a Rota-Baxter operator of weight $1$ on a Lie algebra $\mathfrak{g}$ gives rise to a matched pair of Lie algebras. In the following, we give a generalization of this result in the homotopy context. Let $(\mathcal{G}, \{ \mu_k \}_{k \geq 1})$ be an $L_\infty$-algebra. A {\em strict Rota-Baxter operator of weight $1$} on the $L_\infty$-algebra $(\mathcal{G}, \{ \mu_k \}_{k \geq 1})$ is a degree $0$ linear map $\mathcal{R} : \mathcal{G} \rightarrow \mathcal{G}$ that satisfies
\begin{align}\label{strict-rota}
    \mu_k \big( \mathcal{R} (x_1), \ldots, \mathcal{R} (x_k) \big) = \mathcal{R} \big(  \sum_{p=1}^k \sum_{1 \leq  < i_1 < \cdots < i_p \leq k} \mu_k ( \mathcal{R}(x_1), \ldots, x_{i_1}, \ldots, x_{i_2}, \ldots, x_{i_p} , \ldots, \mathcal{R} (x_k)) \big),
\end{align}
for all $k  \geq 1$ and $x_1, \ldots, x_k \in \mathcal{G}$.

\begin{thm}
    Let $(\mathcal{G}, \{ \mu_k \}_{k \geq 1})$ be an $L_\infty$-algebra and $\mathcal{R} : \mathcal{G} \rightarrow \mathcal{G}$ be a strict Rota-Baxter operator of weight $1$. Then 
    \begin{align*}
        \mathcal{G}_\mathrm{diag} = \{ (x, x) | x \in \mathcal{G} \} ~~~ \text{ and } ~~~ \mathcal{G}_\mathcal{R} = \{ (\mathcal{R}(x), x + \mathcal{R}(x)) | x \in \mathcal{G} \}
    \end{align*}
    are both $L_\infty$-subalgebras of the direct product $L_\infty$-algebra $(\mathcal{G} \oplus \mathcal{G}, \{ \mu_k \oplus \mu_k \}_{k \geq 1} )$. 

    Moreover, if $(\mathcal{G}, \{ \mu_k \}_{k \geq 1})$ is an $L_\infty$-algebra with $\mu_{k \geq 4} = 0$ then $\mathcal{G}_\mathrm{diag}$ and $\mathcal{G}_\mathcal{R}$ constitutes a matched pair of $L_\infty$-algebras. 
\end{thm}

\begin{proof}
    Let $k \geq 1$ and $x_1, \ldots, x_k \in \mathcal{G}$. Then
    \begin{align*}
        (\mu_k \oplus \mu_k) \big(   (x_1, x_1), \ldots, (x_k , x_k )\big) = \big(   \mu_k (x_1, \ldots, x_k ) , \mu_k (x_1, \ldots, x_k) \big) \in \mathcal{G}_\mathrm{diag}.
    \end{align*}
    Hence $\mathcal{G}_\mathrm{diag} \subset \mathcal{G} \oplus \mathcal{G}$ is an $L_\infty$-subalgebra. On the other hand,
    \begin{align*}
        &(\mu_k \oplus \mu_k ) \big( (\mathcal{R}(x_1), x_1 + \mathcal{R}(x_1)), \ldots, (\mathcal{R}(x_k), x_k + \mathcal{R}(x_k)) \big) \\
       & = \big(   \mu_k ( \mathcal{R} (x_1), \ldots, \mathcal{R}(x_k)  ),  \sum_{p=1}^k \sum_{1 \leq  < i_1 < \cdots < i_p \leq k} \mu_k ( \mathcal{R}(x_1), \ldots, x_{i_1}, \ldots, x_{i_2}, \ldots, x_{i_p} , \ldots, \mathcal{R} (x_k)) \\
       & \qquad \qquad \qquad \qquad \qquad \qquad +   \mu_k ( \mathcal{R} (x_1), \ldots, \mathcal{R}(x_k)  )  \big).
    \end{align*}
    This element is in $\mathcal{G}_\mathcal{R}$ once we use the identity (\ref{strict-rota}). Thus, $\mathcal{G}_\mathcal{R} \subset \mathcal{G} \oplus \mathcal{G}$ is also an $L_\infty$-subalgebra.  

    Finally, any element $(x, y) \in \mathcal{G} \oplus \mathcal{G}$ can be uniquely written as
    \begin{align*}
        (x, y) = \big( x - \mathcal{R} (y-x), x - \mathcal{R} (y-x) \big) + \big(  \mathcal{R}(y-x), y-x +  \mathcal{R}(y-x) \big)
    \in \mathcal{G}_\mathrm{diag} \oplus \mathcal{G}_\mathcal{R}.
\end{align*} 
    Thus, we have $\mathcal{G}_\mathrm{diag} \oplus \mathcal{G}_\mathcal{R} = \mathcal{G} \oplus \mathcal{G}$ as a graded vector space. Since $\mu_{k \geq 4} = 0$, it follows that $(\mu_k \oplus \mu_k)_{k \geq 4} = 0$. Hence the final conclusion follows from Remark \ref{remark-last}.
\end{proof}

\section{Matched pairs of skeletal $L_\infty$-algebras}\label{sec7}

In this section, we consider matched pairs of skeletal $L_\infty$-algebras. We show that matched pairs of skeletal $L_\infty$-algebras are closely related to certain $3$-cocycles of matched pairs of Lie algebras. We start with the following definition \cite{baez-crans}.

\begin{defn}\label{2-term}
    A {\em $2$-term $L_\infty$-algebra} is a triple $\mathcal{G} = (\mathfrak{g}_1 \xrightarrow{\mu_1} \mathfrak{g}_0, [~,~]_\mathcal{G}, \mu_3)$ consists of a $2$-term chain complex $\mathfrak{g}_1 \xrightarrow{\mu_1} \mathfrak{g}_0$, a skew-symmetric bilinear map $[~,~]_\mathcal{G} : \mathfrak{g}_i \times \mathfrak{g}_j \rightarrow \mathfrak{g}_{i+j}$ (for $0 \leq i , j, i+j \leq 1$) and a skew-symmetric trilinear operation $\mu_3 : \mathfrak{g}_0 \times \mathfrak{g}_0 \times \mathfrak{g}_0 \rightarrow \mathfrak{g}_1$ that satisfy the following conditions 
    \begin{itemize}
        \item[(i)] $\mu_1 ([x, v]_\mathcal{G}) = [x, \mu_1 (v)]_\mathcal{G},$
        \item[(ii)] $[\mu_1 (u), v]_\mathcal{G} = [u, \mu_1 (v)]_\mathcal{G},$
        \item[(iii)] $\mu_1 ( \mu_3 (x, y, z)) = [x, [y, z]_\mathcal{G} ]_\mathcal{G} + [y, [z, x]_\mathcal{G} ]_\mathcal{G} + [z, [x, y]_\mathcal{G} ]_\mathcal{G},$
        \item[(iv)] $\mu_3 ( x, y, \mu_1(v) ) = [x, [y, v]_\mathcal{G}]_\mathcal{G} + [y, [v, x]_\mathcal{G} ]_\mathcal{G} + [v, [x, y]_\mathcal{G} ]_\mathcal{G},$
        \item[(v)] $[x, \mu_3 (y, z, z')]_\mathcal{G} - [y, \mu_3 (x, z, z')]_\mathcal{G} + [z, \mu_3 (x, y, z')]_\mathcal{G} - [z' , \mu_3 (x, y, z)]_\mathcal{G} = \mu_3 ([x, y]_\mathcal{G}, z, z') \\
        - \mu_3 ([x, z]_\mathcal{G}, y, z') + \mu_3 ([x, z']_\mathcal{G}, y, z) + \mu_3 ([y, z]_\mathcal{G}, x, z') - \mu_3 ([y, z']_\mathcal{G}, x, z) + \mu_3 ([z, z']_\mathcal{G}, x, y),$
    \end{itemize}
    for $x, y, z, z' \in \mathfrak{g}_0$ and $u, v \in \mathfrak{g}_1$.
\end{defn}

Note that $2$-term $L_\infty$-algebras are precisely those $L_\infty$-algebras whose underlying graded vector space is concentrated in arities $0$ and $1$.
A $2$-term $L_\infty$-algebra $\mathcal{G} = (\mathfrak{g}_1 \xrightarrow{\mu_1} \mathfrak{g}_0, [~,~]_\mathcal{G}, \mu_3)$ is called {\em skeletal} if $\mu_1 = 0$. In \cite{baez-crans} Baez and Crans showed that skeletal $2$-term $L_\infty$-algebras (also called skeletal $L_\infty$-algebras) are closely related to $3$-cocycles of Lie algebras. More precisely, let $\mathcal{G} = (\mathfrak{g}_1 \xrightarrow{0} \mathfrak{g}_0, [~,~]_\mathcal{G}, \mu_3)$ be a skeletal $L_\infty$-algebra. Then it follows from the condition (iii) of Definition \ref{2-term} that the vector space $\mathfrak{g}_0$ equipped with the skew-symmetric bracket $[~,~]_\mathcal{G}: \mathfrak{g}_0 \times \mathfrak{g}_0 \rightarrow \mathfrak{g}_0$ is a Lie algebra. Moreover, the condition (iv) of Definition \ref{2-term}  implies that the vector space $\mathfrak{g}_1$ with the bilinear map $[~,~]_\mathcal{G} : \mathfrak{g}_0 \times \mathfrak{g}_1 \rightarrow \mathfrak{g}_1$ forms a representation of the Lie algebra $(\mathfrak{g}_0 , [~,~]_\mathcal{G})$. Finally, the condition (v) simply means that
\begin{align*}
    \delta_\mathrm{CE} (\mu_3) = 0,
\end{align*}
where $\delta_\mathrm{CE}$ is the Chevalley-Eilenberg coboundary operator of the Lie algebra $(\mathfrak{g}_0, [~,~]_\mathcal{G})$ with coefficients in the representation $\mathfrak{g}_1$. Thus, we obtain a triple $(\mathfrak{g}_0, \mathfrak{g}_1, \mu_3)$ consisting of a Lie algebra $\mathfrak{g}_0$, a representation $\mathfrak{g}_1$ and a $3$-cocycle $\mu_3$. Conversely, if we have a triple $(\mathfrak{g}, V, \theta)$ in which $\mathfrak{g}$ is a Lie algebra, $V= (V, \rho)$ is a representation and $\theta \in Z^3 (\mathfrak{g}, V)$ is a $3$-cocycle then $\mathcal{G} := (V \xrightarrow{0} \mathfrak{g}, [~,~]_\mathcal{G}, \theta)$ is a skeletal $L_\infty$-algebra, where the bilinear operation $[~,~]_\mathcal{G}$ is given by
\begin{align*}
    [x, y]_\mathcal{G} := [x, y]_\mathfrak{g}, \quad [x, v]_\mathcal{G} = - [v, x]_\mathcal{G} := \rho_x v, \text{ for } x, y \in \mathfrak{g} \text{ and } v \in V.
\end{align*}

\medskip

Since we will mostly focus on skeletal $L_\infty$-algebras, we only define representations of them. Suppose $(\mathfrak{g}_1 \xrightarrow{0} \mathfrak{g}_0 , [~,~]_\mathcal{G}, \mu_3)$ is a skeletal $L_\infty$-algebra. A {\em representation} of it is a $2$-term trivial complex $V_1 \xrightarrow{0} V_0$ equipped with a bilinear map $\rho_2 : \mathfrak{g}_i \times V_j \rightarrow V_{i+j}$ (for $0 \leq i, j, i+j \leq 1$) and a trilinear map $\rho_3 : \mathfrak{g}_0 \times \mathfrak{g}_0 \times V_0 \rightarrow V_1$ subject to satisfy the following identities:
\begin{align}
   &\rho_3 (x, y, v) = - \rho_3 (y, x, v), \label{skel-rep1}\\
   & \rho_2 (x, \rho_2 (y, v)) - \rho_2 (y, \rho_2 (x, v)) - \rho_2 ([x, y]_\mathcal{G}, v) = 0,\\
    &\rho_2 (x, \rho_2 (y, \overline{v})) - \rho_2 (y, \rho_2 (x, \overline{v})) - \rho_2 ([x, y]_\mathcal{G}, \overline{v}) = 0,\\
   & \rho_2 (x, \rho_3 (y, z, v) ) - \rho_2 (y, \rho_3 (x, z, v)) + \rho_2 (z, \rho_3 (x, y, v)) + \rho_2 (\mu_3 (x, y, z), v) = \rho_3 ([x, y]_\mathcal{G}, z, v) \label{skel-rep4}\\
   & \qquad - \rho_3 ([x, z]_\mathcal{G}, y, v) + \rho_3 (y, z, \rho_2 (x, v)) + \rho_3 ([y, z]_\mathcal{G}, x, v) - \rho_3 (x, z, \rho_2 (y, v)) + \rho_3 (x, y, \rho_2 (z, v)), \nonumber
\end{align}
for $x, y, z \in \mathfrak{g}_0$, $v \in V_0$ and $ \overline{v} \in V_1$. A representation as above may be denoted by the triple $(V_1 \xrightarrow{0} V_0, \rho_2, \rho_3)$. It follows from the above conditions that any skeletal $L_\infty$-algebra $(\mathfrak{g}_1 \xrightarrow{0} \mathfrak{g}_0, [~,~]_\mathcal{G}, \mu_3)$ can be regarded as a representation of itself.

In the following, we explicitly describe matched pairs of skeletal $L_\infty$-algebras and give a cohomological characterization.

\begin{defn}
    A {\em matched pair of skeletal $L_\infty$-algebras} is given by a tuple
    \begin{align*}
        \big(  (\mathfrak{g}_1 \xrightarrow{0} \mathfrak{g}_0, [~,~]_\mathcal{G}, \mu_3) , (\mathfrak{h}_1 \xrightarrow{0} \mathfrak{h}_0, [~,~]_\mathcal{H}, \nu_3), \rho_2, \rho_3, \psi_2, \psi_3   \big)
    \end{align*}
    in which $ \mathcal{G} := (\mathfrak{g}_1 \xrightarrow{0} \mathfrak{g}_0, [~,~]_\mathcal{G}, \mu_3)$ and $\mathcal{H} := (\mathfrak{h}_1 \xrightarrow{0} \mathfrak{h}_0, [~,~]_\mathcal{H}, \nu_3)$ are skeletal $L_\infty$-algebras,
\begin{itemize}
  \item the maps $\rho_2 : \mathfrak{g}_i \times \mathfrak{h}_j \rightarrow \mathfrak{h}_{i+j}$ and $\rho_3 : \mathfrak{g}_0 \times \mathfrak{g}_0 \times \mathfrak{h}_0 \rightarrow \mathfrak{h}_1$ make the triple $(\mathfrak{h}_1 \xrightarrow{0} \mathfrak{h}_0 , \rho_2, \rho_3)$ into a representation of the skeletal $L_\infty$-algebra $\mathcal{G}$,

   \item the maps $\psi_2 : \mathfrak{h}_i \times \mathfrak{g}_j \rightarrow \mathfrak{g}_{i+j}$ and $\psi_3 : \mathfrak{h}_0 \times \mathfrak{h}_0 \times \mathfrak{g}_0 \rightarrow \mathfrak{g}_1$ make the triple $(\mathfrak{g}_1 \xrightarrow{0} \mathfrak{g}_0 , \psi_2, \psi_3)$ into a representation of the skeletal $L_\infty$-algebra $\mathcal{H}$,

   \item the identities are hold: for $x, y \in \mathfrak{g}_0$, $v \in \mathfrak{g}_1$, $h, k \in \mathfrak{h}_0$ and $w \in \mathfrak{h}_1$,
    \begin{align}
            \rho_2 ( [x, v]_\mathcal{G} ,h) =~& \rho_2 (x, \rho_2 (v, h)) - \rho_2 (v, \rho_2 (x, h)), \label{skew-l-repp1}\\
            \psi_2 ([h, w]_\mathcal{H}, x) =~& \psi_2 (h, \psi_2 (w, x)) - \psi_2 (w, \psi_2 (h, x)), \\
            \rho_2 (x, [h, w]_\mathcal{H}) =~& [\rho_2 (x, h), w]_\mathcal{H} + [h, \rho_2 (x, w)]_\mathcal{H} + \rho_2 (\psi_2 (w, x), h) - \rho_2 (\psi_2 (h, x), w), \\
            \rho_2 (v, [h, k]_\mathcal{H}) =~& [\rho_2 (v, h), k]_\mathcal{H} + [h, \rho_2 (v, k)]_\mathcal{H} + \rho_2 (\psi_2 (k, v), h) - \rho_2 (\psi_2 (h, v), k), \\
            \psi_2 (h, [x, v]_\mathcal{G}) =~& [\psi_2 (h, x), v]_\mathcal{G} + [x, \psi_2 (h, v)]_\mathcal{G} + \psi_2 (\rho_2 (v, h), x) - \psi_2 (\rho_2 (x, h), v), \\
            \psi_2 (w, [x, y]_\mathcal{G}) =~& [\psi_2 (w, x), y]_\mathcal{G} + [x, \psi_2 (w, y)]_\mathcal{G} + \psi_2 (\rho_2 (y, w), x) - \psi_2 (\rho_2 (x, w), y), \label{skew-l-repp6}
        \end{align}

   \item the following compatibilities are hold: for all $x, y, z \in \mathfrak{g}_0$ and $h, k, k' \in \mathfrak{h}_0$,
   \begin{align}
      & [x, \psi_3 (h, k, y)]_\mathcal{G} - [y, \psi_3 (h, k, x)]_\mathcal{G} - \psi_3 (h, k, [x, y]_\mathcal{G} ) - \psi_3 (\rho_2 (x, h), k, y) \nonumber \\
      & \qquad + \psi_3 (\rho_2 (x, k), h, y) + \psi_3 (\rho_2 (y, h), k, x) - \psi_3 (\rho_2 (y, k), h, x) = 0, \label{skel1}\\
      & \rho_2 (x, \mu_3 (h, k, k')) + \rho_2 ( \psi_3 (k, k', x), h) - \rho_2 (\psi_3 (h, k', x), k) + \rho_2 ( \psi_3 (h, k, x), k') \nonumber  \\
      & \qquad - \nu_3 (\rho_2 (x, h), k, k') + \nu_3 (\rho_2 (x, k), h, k') - \nu_3 (\rho_2 (x, k'), h, k) = 0, \label{skel2}
   \end{align}
      \begin{align}
      & [h, \rho_3 (x, y, k)]_\mathcal{H} - [k, \rho_3 (x, y, h)]_\mathcal{H} - \rho_3 (x, y, [h, k]_\mathcal{H} ) - \rho_3 (\psi_2 (h, x), y, k) \nonumber \\
      & \qquad + \rho_3 (\psi_2 (h, y), x, k) + \rho_3 (\psi_2 (k, x), y, h) - \rho_3 (\psi_2 (k, y), x, h) = 0, \label{skel3} \\
      & \psi_2 (h, \nu_3 (x, y, z)) + \psi_2 ( \rho_3 (y, z, h), x) - \psi_2 (\rho_3 (x, z, h), y) + \psi_2 ( \rho_3 (x, y, h), z) \nonumber  \\
      & \qquad - \mu_3 (\psi_2 (h, x), y, z) + \mu_3 (\psi_2 (h, y), x, z) - \mu_3 (\psi_2 (h, z), x, y) = 0. \label{skel4}
   \end{align}
   \end{itemize}
\end{defn}

We are now in a position to prove our final theorem that characterizes matched pairs of skeletal $L_\infty$-algebras.

\begin{thm}
    There is a one-to-one correspondence between the following sets
    \begin{align*}
        \bigg\{  \substack{\text{the set of all matched pairs of} \\ \text{ skeletal } L_\infty\text{-algebras }}  \bigg\} 
        \longleftrightarrow \bigg\{  \substack{ \text{the set of triples } \big( (\mathfrak{g}, \mathfrak{h}, \rho, \psi) , (V, W, \alpha, \beta ), (F_1, 0, F_3) \big), \\ \text{where }  (\mathfrak{g}, \mathfrak{h}, \rho, \psi) \text{is a matched pair of Lie algebras,} \\ (V, W, \alpha, \beta) \text{ is a representation and} \\ (F_1, 0, F_3) \text{ is a 3-cocycle} }  \bigg\}.
    \end{align*}
\end{thm}

\begin{proof}
Let $ \big(  (\mathfrak{g}_1 \xrightarrow{0} \mathfrak{g}_0, [~,~]_\mathcal{G}, \mu_3) , (\mathfrak{h}_1 \xrightarrow{0} \mathfrak{h}_0, [~,~]_\mathcal{H}, \nu_3), \rho_2, \rho_3, \psi_2, \psi_3   \big)$ be a matched pair of skeletal $L_\infty$-algebras. Note that
\begin{align*}
   & \mathcal{G} = (\mathfrak{g}_1 \xrightarrow{0} \mathfrak{g}_0, [~,~]_\mathcal{G}, \mu_3) \text{ is a skeletal } L_\infty\text{-algebra} \\
   & \qquad \Longleftrightarrow \mathfrak{g}_0 \text{ is a Lie algebra, } \mathfrak{g}_1 \text{ is a representation of the Lie algebra } \mathfrak{g}_0 \text{ and } \mu_3 \in \mathrm{Hom}(\wedge^3 \mathfrak{g}_0 , \mathfrak{g}_1) \\
   & \qquad \qquad \text{ is a Chevalley-Eilenberg $3$-cocycle,}\\
   & \mathcal{H} = (\mathfrak{h}_1 \xrightarrow{0} \mathfrak{h}_0, [~,~]_\mathcal{H}, \nu_3) \text{ is a skeletal } L_\infty\text{-algebra} \\
   & \qquad \Longleftrightarrow \mathfrak{h}_0 \text{ is a Lie algebra, } \mathfrak{h}_1 \text{ is a representation of the Lie algebra } \mathfrak{h}_0 \text{ and } \nu_3 \in \mathrm{Hom}(\wedge^3 \mathfrak{h}_0 , \mathfrak{h}_1) \\
   & \qquad \qquad \text{ is a Chevalley-Eilenberg $3$-cocycle}.
\end{align*}
We define two maps $\mu_3 \ltimes \rho_3 \in C^{2|0} (\mathfrak{g}_0, \mathfrak{h}_0; \mathfrak{g}_1, \mathfrak{h}_1)$ and $ \psi_3 \rtimes \nu_3 \in C^{0|2} (\mathfrak{g}_0, \mathfrak{h}_0; \mathfrak{g}_1, \mathfrak{h}_1)$ by
\begin{align*}
    (\mu_3 \ltimes \rho_3) ((x,h), (y, k), (z, k'))  :=~& \big(  \mu_3 (x, y, z), \rho_3 (x, y, k') +  \rho_3 (y, z, h) + \rho_3 (z, x, k) \big), \\
    (\psi_3 \rtimes \nu_3) ((x,h), (y, k), (z, k')) := ~&\big( \psi_3 (h, k, z)  + \psi_3 (k, k', x) + \psi_3 (k', h, y), \nu_3 (h, k, k') \big),
\end{align*}
for $(x, h), (y, k), (z, k') \in \mathfrak{g}_0 \oplus \mathfrak{h}_0$. With these notations, we observe that
   \begin{align*}
   &(\mathfrak{h}_1 \xrightarrow{0} \mathfrak{h}_0, \rho_2, \rho_3) \text{ is a representation of the skeletal } L_\infty\text{-algebra } \mathcal{G} \\
   & \qquad \Longleftrightarrow \mathfrak{h}_0, \mathfrak{h}_1 \text{ both are representations of the Lie algebra }\mathfrak{g}_0 \text{ and } (\delta^{[~,~]_\mathcal{G} \ltimes \rho_2} (\mu_3 \ltimes \rho_3) )(x, y, z, h) = 0, \\
   & \qquad \qquad \quad (\text{follows from }(\ref{skel-rep1})-(\ref{skel-rep4})) \\
   &(\mathfrak{g}_1 \xrightarrow{0} \mathfrak{g}_0, \psi_2, \psi_3) \text{ is a representation of the skeletal } L_\infty\text{-algebra } \mathcal{H} \\
   & \qquad \Longleftrightarrow \mathfrak{g}_0, \mathfrak{g}_1 \text{ both are representations of the Lie algebra }\mathfrak{h}_0 \text{ and } (\delta^{\psi_2 \rtimes [~,~]_\mathcal{H}} (\psi_3 \rtimes \nu_3) )( h, k, k', x) = 0.
\end{align*}
The identities (\ref{skew-l-repp1})-(\ref{skew-l-repp6}) imply that $(\mathfrak{g}_1, \mathfrak{h}_1, \rho_2, \psi_2)$ is a representation of the matched pair of Lie algebras $(\mathfrak{g}_0, \mathfrak{h}_0, \rho_2, \psi_2 )$. Finally, the compatibility conditions (\ref{skel1}), (\ref{skel2}) are equivalent to $\delta^{[~,~]_\mathcal{G} \ltimes \rho_2} (\psi_3 \rtimes \nu_3) = 0$ and the conditions (\ref{skel3}), (\ref{skel4}) are equivalent to $\delta^{\psi_2 \rtimes [~,~]_\mathcal{H}} (\mu_3 \ltimes \rho_3) = 0$. Therefore, we obtain the required triple
\begin{align*}
    \big((\mathfrak{g}_0, \mathfrak{h}_0, \rho_2, \psi_2 ), (\mathfrak{g}_1, \mathfrak{h}_1, \rho_2, \psi_2 ) ,(\mu_3 \ltimes \rho_3, 0 , \psi_3 \rtimes \nu_3) \big).
\end{align*}
The converse part is precisely the reverse calculation.
\end{proof}

\medskip

\medskip

\noindent  {\bf Acknowledgements.} The authors would like to thank the Department of Mathematics, IIT Kharagpur for providing the beautiful academic atmosphere where the research has been carried out.

\medskip

\noindent {\bf Data Availability Statement.} Data sharing does not apply to this article as no new data were created or analyzed in this study.

\end{document}